\documentclass[11pt,reqno,a4paper]{amsart}
\usepackage[dvipsnames, table]{xcolor}
\usepackage{amsthm}
\theoremstyle{plain}
\usepackage{amssymb}
\usepackage{marvosym}
\usepackage{bm}
\usepackage{mathrsfs}
\usepackage{enumerate}  
\usepackage{mathtools}
\usepackage{tikz-cd}
\usepackage{tikz}
\usepackage{graphicx}
\usepackage{float}
\usepackage[titletoc,title]{appendix}
\usepackage{marginnote}
\usepackage{todonotes}
\usepackage{etoolbox}
\usepackage[all]{xy}
\makeatletter
\patchcmd{\Ginclude@eps}{"#1"}{#1}{}{}
\makeatother
\usepackage[outdir=./]{epstopdf}
\usepackage[numbers]{natbib}
\usepackage[utf8]{inputenc}
\usepackage[english]{babel}
\usepackage{changepage}
\usepackage{makecell}
\usepackage{enumitem}

\usepackage[colorlinks,citecolor=blue]{hyperref}

\definecolor{lightblue}{HTML}{1F88CD}
\definecolor{lightgrey}{HTML}{727272}
\definecolor{lightblue2}{HTML}{009EC1}
\definecolor{mypink}{HTML}{FD00B0}
\definecolor{lightred}{HTML}{ff4d4d}

\usepackage{hyperref}
\hypersetup{
	colorlinks=true,
    linkcolor={OliveGreen},
    citecolor={lightred},
	urlcolor={black}
}

\newtheorem*{theorem*}{Theorem}
\newtheorem{theorem}{Theorem}[section]
\newtheorem{corollary}[theorem]{Corollary}
\newtheorem{lemma}[theorem]{Lemma}
\newtheorem{conjecture}[theorem]{Conjecture}

\newtheorem{proposition}[theorem]{Proposition}
\theoremstyle{definition}
\newtheorem{example}[theorem]{Example}
\theoremstyle{definition}
\newtheorem{definition}[theorem]{Definition}
\theoremstyle{definition}
\newtheorem{remark}[theorem]{Remark}
\theoremstyle{definition}

\theoremstyle{definition}

\theoremstyle{definition}

\theoremstyle{definition}

\theoremstyle{definition}

\theoremstyle{definition}
\newtheorem{question!}[theorem]{Question!}
\theoremstyle{definition}

\makeatletter
\newcommand*\sbt{\mathpalette\sbt@{.75}}
\newcommand*\sbt@[2]{\mathbin{\vcenter{\hbox{\scalebox{#2}{$\m@th#1\bullet$}}}}}
\makeatother

\newcommand{\ra}{\rightarrow}
\newcommand{\xra}{\xrightarrow}

\newcommand{\sst}{\subset}

\newcommand{\bR}{\bm{\mathrm{R}}}
\newcommand{\bL}{\bm{\mathrm{L}}}

\newcommand{\D}{\mathrm{D}}

\newcommand{\CC}{\mathbb{C}}

\newcommand{\ch}{\mathrm{ch}}

\newcommand{\pr}{\mathrm{pr}}

\renewcommand{\Re}{\operatorname{Re}}
\renewcommand{\Im}{\operatorname{Im}}

\DeclareMathOperator{\Ext}{Ext}

\DeclareMathOperator{\Hom}{Hom}
\DeclareMathOperator{\RHom}{RHom}

\DeclareMathOperator{\ext}{ext}

\DeclareMathOperator{\cone}{cone}

\newcommand{\cC}{\mathcal{C}}
\newcommand{\cA}{\mathcal{A}}

\newcommand{\cU}{\mathcal{U}}
\newcommand{\cH}{\mathcal{H}}

\newcommand{\cQ}{\mathcal{Q}}
\newcommand{\Ku}{\mathcal{K}u}

\newcommand{\cD}{\mathcal{D}}

\newcommand{\cM}{\mathcal{M}}

\DeclareMathOperator{\oh}{\mathcal{O}}

\usetikzlibrary{decorations.pathmorphing}
\begin{document}

\title[]{A Moduli theoretic approach to Lagrangian subvarieties of hyperkähler varieties: Examples }


\address{Beijing International Center for Mathematical Research, Peking University, Beijing, China}
\email{hfguo@pku.edu.cn}

\address{College of Mathematics, Sichuan University, Chengdu, Sichuan Province 610064 P. R.
China}
\email{zhiyuliu@stu.scu.edu.cn}

\address{School of Mathematics, The University of Edinburgh, James Clerk Maxwell Building, Kings Buildings, Edinburgh, United Kingdom, EH9 3FD}
\email{Shizhuo.Zhang@ed.ac.uk}

\author{Hanfei Guo, Zhiyu Liu, Shizhuo Zhang}
\address{}
\email{}

\begin{abstract}
We propose two conjectures on a moduli theoretic approach to constructing Lagrangian subvarieties of hyperkähler varieties arising from the Kuznetsov components of cubic fourfolds or Gushel--Mukai fourfolds. Then we verify the conjectures in several cases, recovering classical examples. As a corollary, we confirm a conjecture of O'Grady in several instances on the existence of Lagrangian covering families for hyperkähler varieties.
\end{abstract}


\maketitle

\setcounter{tocdepth}{1}








\section{Introduction}
\subsection{K3 categories and hyperkähler varieties}
This paper is the first one of our series of work on a systematic way of constructing Lagrangian subvarieties for hyperk\"{a}hler varieties as moduli spaces of stable objects on a non-commutative K3 surface. 

A triangulated category is called a non-commutative K3 surface (or a K3 category) if it has the same Serre functor and Hochschild cohomology as the derived category of a K3 surface. Examples of non-commutative K3 surfaces are the Kuznetsov components in the semi-orthogonal decomposition of the derived categories of certain Fano fourfolds, such as cubic fourfolds and Gushel--Mukai fourfolds.

\subsubsection{Cubic fourfolds}
Let $X$ be a cubic fourfold. Its semi-orthogonal decomposition is given by $$D^b(X)=\langle\Ku(X),\oh_X,\oh_X(H),\oh_X(2H)\rangle=\langle\oh_X(-H),\Ku(X),\oh_X,\oh_X(H)\rangle.$$
We define the projection functor to the Kuznetsov component  $\pr_X:=\bR_{\oh_X(-H)}\bL_{\oh_X}\bL_{\oh_X(H)}$. There is a rank two lattice in the numerical Grothendieck group $\mathcal{N}(\Ku(X))$ generated by $$\Lambda_1=3-H-\frac{1}{2}H^2+\frac{1}{2}L+\frac{3}{8}P,\quad\Lambda_2=-3+2H-L,$$ 
over which the Euler pairing is of the form
\begin{equation}
\left[               
\begin{array}{cc}   
-2 & 1 \\  
1 & -2\\
\end{array}
\right].
\end{equation}

Let $v(a, b):=a\Lambda_1+b\Lambda_2$ be a primitive class. In \cite{bayer2017stability}, the authors construct a family of stability conditions on $\Ku(X)$. Then we can construct the Bridgeland moduli space $\mathcal{M}^X_{\sigma}(a,b)$ of $\sigma$-stable objects in $\Ku(X)$ with character $v(a, b)$. If $\sigma$ is generic with respect to $v(a, b)$, the moduli space $\mathcal{M}^X_{\sigma}(a,b)$ is shown to be a smooth projective hyperkähler variety of dimension $2(a^2+b^2-ab+1)$ by \cite[Theorem 29.2]{BLMNPS21}.

\subsubsection{Gushel--Mukai fourfolds}
Let $X$ be an ordinary Gushel--Mukai (GM) fourfold, the semi-orthogonal decomposition of $\D^b(X)$ is given by 
$$D^b(X)=\langle\oh_X(-H),\cU,\Ku(X),\oh_X,\cU^{\vee}\rangle=\langle\Ku(X),\oh_X,\cU^{\vee},\oh_X(H),\cU^{\vee}(H)\rangle.$$
We define the projection functor $\mathrm{pr}_X:=\bR_{\cU}\bR_{\oh_X(-H)}\bL_{\oh_X}\bL_{\cU^{\vee}}$.
There is a rank two lattice in~$\mathcal{N}(\Ku(X))$ generated by $$\Lambda'_1=-2+(H^2-\Sigma)-\frac{1}{2}P,
\quad\Lambda'_2=-4+2H-\frac{5}{3}L,$$ 
whose Euler pairing is 
\begin{equation}
\left[               
\begin{array}{cc}   
-2 & 0 \\  
0 & -2\\
\end{array}
\right].
\end{equation}
Analogously, let $v'(a,b):=a\Lambda'_1+b\Lambda'_2$ be a primitive vector. 
In \cite{perry2019stability}, the authors construct a family of  stability conditions on $\Ku(X)$. It is shown by \cite[Theorem 1.5]{perry2019stability} that if $\sigma'$ is generic with respect to $v'$,  
the moduli space $\mathcal{M}^X_{\sigma'}(a,b)$ is a smooth projective hyperkähler variety of dimension $2(a^2+b^2+1)$.

\subsection{Fano threefolds and Lagrangian subvarieties}
On the other hand, a hyperplane section of these Fano fourfolds (i.e., a cubic threefold or a GM threefold) also admits a semi-orthogonal decomposition.
\subsubsection{Cubic threefolds}
For a cubic threefold $Y$, we have a semi-orthogonal decomposition
\[D^b(Y)=\langle \Ku(Y), \oh_Y, \oh_Y(H)\rangle.\]
Moreover, $\mathcal{N}(\Ku(Y))$ is a rank two lattice $\langle\lambda_1,\lambda_2\rangle$, where
$$\lambda_1=2-H-\frac{1}{2}L+\frac{1}{2}P ,\quad\lambda_2=-1+H-\frac{1}{2}L-\frac{1}{2}P$$
and the Euler pairing is 
\begin{equation}
\left[               
\begin{array}{cc}   
-1 & 1 \\  
0 & -1\\
\end{array}
\right].
\end{equation}
Let $w(a,b):=a\lambda_1+b\lambda_2$. According to \cite[Theorem 1.2]{PY20}, for a Serre-invariant stability condition $\tau$ on $\Ku(Y)$, if $\mathcal{M}^Y_{\tau}(a,b)$ is non-empty, it is a smooth variety of dimension $a^2+b^2-ab+1$. 

\subsubsection{Gushel--Mukai threefolds}
For a GM threefold $Y$, we have a semi-orthogonal decomposition
\[D^b(Y)=\langle \Ku(Y), \oh_Y, \cU_Y^{\vee}\rangle.\]
Moreover, $\mathcal{N}(\Ku(Y))$ is a rank two lattice generated by 
$$\lambda_1'=-1+2L,\quad\lambda_2'=2-H+\frac{5}{6}P,$$
with the Euler pairing
\begin{equation}
\left[               
\begin{array}{cc}   
-1 & 0 \\  
0 & -1\\
\end{array}
\right].
\end{equation}
Let $w'(a,b):=a\lambda_1'+b\lambda_2'$. In \cite{bayer2017stability}, the authors construct a family of stability conditions on $\Ku(Y)$. However, for a Serre-invariant stability condition $\tau'$, the moduli space $\mathcal{M}^Y_{\tau'}(a,b)$ may not be smooth of the expected dimension $a^2+b^2+1$.

It is easy to see that $w(a,b)$ and $w'(a,b)$ is half the dimension of $v(a,b)$ and $v'(a,b)$ respectively. For the rest of our paper, let $\tau$ and $\tau'$ denotes the Serre-invariant stability condition on the Kuznetsov component of two Fano threefolds respectively. We always assume that $\sigma$ is generic with respect to $v(a,b)$ and $\sigma'$ is generic with respect to $v'(a,b)$.

The main motivation of our article comes from the following classical example.

\begin{example}\label{motivation_example}
Let $S$ be a K3 surface, if $S$ contains a curve $C$, then $\mathrm{Sym}^n(C)$ is a Lagrangian subvariety of the Hilbert scheme of points $S^{[n]}$, which is a hyperkähler variety.
\end{example}

Indeed, the Kuznetsov component of most prime Fano threefolds $Y$ can be regarded as a non-commutative curve in the sense that the  numerical Grothendieck group is the same as that of a curve. In addition, it is shown by \cite[Theorem 4.20]{JLLZ2021gushelmukai} and \cite[Theorem 3.2]{FeyzbakhshPertusi2021stab} that the Serre-invariant stability condition on $\Ku(Y)$ is unique, which is true for every curve of genus $g>0$ (cf. \cite{Mac07}).
Inspired by the construction in Example~\ref{motivation_example}, we aim to explore the connection between Bridgeland moduli spaces over a non-commutative curve and Bridgeland moduli spaces over a non-commutative K3 surface. More precisely, we define a functor 
$$\pr_X\circ j_*:\Ku(Y)\rightarrow\Ku(X)$$ and expect that it induces a Lagrangian embedding to a hyperkähler variety, at least generically.

\begin{conjecture}
\label{main_conjecture_1}
Let $X$ be a cubic fourfold and $j:Y\hookrightarrow X$ be its smooth hyperplane section.
If the moduli space $\cM^Y_{\tau}(a,b)$ is non-empty,
then the functor $$\pr_X\circ j_*:\Ku(Y)\rightarrow\Ku(X)$$ induces a rational map of moduli spaces $$\cM^Y_{\tau}(a,b)\dashrightarrow\cM^X_{\sigma}(a,b)$$  
such that $\cM^Y_{\tau}(a,b)$ is birational to its image $L$. Furthermore, $L$ is Lagrangian in $\cM^X_{\sigma}(a,b)$.
\end{conjecture}

\begin{conjecture}
\label{main_conjecture_2}
Let $X$ be a GM fourfold and $j:Y\hookrightarrow X$ be its smooth hyperplane section. Then the functor $$\pr_X\circ j_*:\Ku(Y)\rightarrow\Ku(X)$$ induces a rational map of moduli spaces $$\cM^Y_{\tau'}(a,b)\dashrightarrow\cM^X_{\sigma'}(a,b)$$  
such that $\cM^Y_{\tau'}(a,b)$ is birational to its image $L$. Furthermore, $L$ is Lagrangian in $\cM^X_{\sigma'}(a,b)$.
\end{conjecture}

\begin{remark}\label{non_empty_moduli space}
By \cite[Theorem 1.5(1)]{perry2019stability} and \cite[Theorem 29.2(1)]{BLMNPS21}, both moduli spaces~$\mathcal{M}^{X}_{\sigma}(a,b)$ and $\mathcal{M}^X_{\sigma'}(a,b)$ are non-empty. In the upcoming work \cite{ppzEnriques2021}, the authors show that the moduli space $\mathcal{M}^Y_{\tau'}(a,b)$ is non-empty.
\end{remark}


\begin{remark}
It is easy to observe that the induced map 
$$[\pr_X\circ j_*]:\mathcal{N}(\Ku(Y))\rightarrow\mathcal{N}(\Ku(X))$$ maps $w(a,b)$ to $v(a,b)$ and $w'(a,b)$ to $v'(a,b)$. 
Note that once we have such a rational map as in Conjecture \ref{main_conjecture_1}, the image $L$ is automatically Lagrangian. The analogous statement holds for Conjecture \ref{main_conjecture_2}. See Lemma \ref{lag_lem}.
\end{remark}

\subsection{Main results}
In the present paper, we verify Conjecture~\ref{main_conjecture_1} when  $(a,b)=(1,1)$ and $(2,1)$ and Conjecture~\ref{main_conjecture_2} when $(a,b)=(1,0)$. These specific cases will recover several classical examples of Lagrangian subvarieties for hyperkähler varieties.

\begin{theorem}
\label{main_theorem_1}
\leavevmode\begin{enumerate}
\item Conjecture \ref{main_conjecture_1} holds for $(a,b)=(1,1)$. As a result, the Fano surface of lines $F(Y)$ is a Lagrangian subvariety of the Fano variety of lines $F(X)$, which is observed in \cite{voisin1992stabilite}.
    
\item Conjecture \ref{main_conjecture_1} holds for $(a,b)=(2,1)$ when $X$ is a general cubic fourfold. As a result, the image of the twisted cubics on a cubic threefold under the two-step contraction is a Lagrangian subvariety of the LLSvS eightfold, which is observed in \cite[Proposition 6.9]{shinder2017geometry}. 

\end{enumerate}
\end{theorem}

\begin{theorem}
\label{main_theorem_2}
\leavevmode
    Conjecture \ref{main_conjecture_2} holds for $(a,b)=(1,0)$ when $X$ is a very general GM fourfold. As a result, the double dual EPW surface associated with $Y$ is a Lagrangian subvariety of the double dual EPW sextic associated with $X$, as observed in \cite[Proposition 5.1]{iliev2011fano}. 
  
\end{theorem}

In the sequel \cite{FLZlag2022} of the current paper, we study Conjectures \ref{main_conjecture_1} and \ref{main_conjecture_2} in a general situation.

\begin{remark}
In Conjecture~\ref{main_conjecture_1} and \ref{main_conjecture_2}, we only speculate that the functor $\mathrm{pr}_X\circ j_*$ induces a rational map between the moduli spaces, birational to its image. However, the results in Theorem~\ref{main_theorem_1} and \ref{main_theorem_2} are much stronger. More precisely, in our cases, the functor $\mathrm{pr}_X\circ j_*$ actually induces a closed Lagrangian embedding to a hyperkähler variety. We propose the weaker conjectures mainly for technical reasons. Indeed, it is usually difficult to verify the stability of objects after taking the push-forward functor. An optimistic idea is to verify the stability generically. We appreciate Chunyi Li for reminding us this technical obstacle.
\end{remark}

Now we introduce a notion called a Lagrangian covering family for hyperkähler varieties, which is proposed by O'Grady.  Furthermore, he conjectures that all projective hyperkähler varieties admit a Lagrangian covering family. Lagrangian covering families, which can be seen as a generalized version of Lagrangian fibrations, are studied in different contexts for hyperkähler varieties. For example, the surface decomposable property  \cite{voisin2018triangle}, the Lefschetz standard conjecture \cite{Voi21} and some cohomological criterion \cite{Bai22}.

\begin{definition}(Lagrangian covering family)
Let $M$ be a hyperkähler variety of dimension~$2n$. A Lagrangian covering family for $M$ is a closed subscheme $\cU\subset M\times B$, pure of dimension~$\mathrm{dim}(B)+n$ with the projections
\[\begin{tikzcd}
\cU \arrow[r, "\pi_M"] \arrow[d, "\pi_B"'] & M\\
B                                     &  
\end{tikzcd}\]
such that
\leavevmode\begin{enumerate}
\item the general fiber of $\pi_B$ is a Lagrangian subvariety of $M$.
\item $\pi_M(\cU)=M$.
\end{enumerate}
\end{definition}

\begin{conjecture}
\label{Conjecture_Lag_covering}(O'Grady)
Any projective hyperkähler variety admits a Lagrangian covering family.
\end{conjecture}


O'Grady provides a number of examples satisfying Conjecture~\ref{Conjecture_Lag_covering} such as the Fano variety of lines and the LLSvS eightfold associated with a cubic fourfold and the double EPW sextic associated with a GM fourfold.
We explain these examples in detail which are probably known to the experts. As a corollary of Theorem~\ref{main_theorem_1} and \ref{main_theorem_2}, we offer some examples of Lagrangian covering families for hyperkähler varieties when they are Bridgeland moduli spaces over K3 categories. 

\begin{corollary}\label{main_theorem_3}
\leavevmode\begin{enumerate}
    \item Let $X$ be a general cubic fourfold. Then the hyperkähler varieties $\mathcal{M}_{\sigma}^X(1,1)$ and $\mathcal{M}_{\sigma}^X(2,1)$ admit a Lagrangian covering family. 
    \item Let $X$ be a very general GM fourfold. Then the hyperkähler varieties $\mathcal{M}^X_{\sigma'}(1,0)$ and~$\mathcal{M}^X_{\sigma'}(0,1)$ admit a Lagrangian covering family. 
\end{enumerate}

\end{corollary}

\subsection{Notation and conventions}
\begin{itemize}
\item We work over $k=\mathbb{C}$.
\item The term K3 surface means a smooth projective K3 surface.
\item We denote the bounded derived category of coherent sheaves on a smooth variety~$X$ by~$\D^b(X)$. The derived dual functor $R\mathcal{H}om_X(-,\oh_X )$ on $\D^b(X)$ is denoted by~$\mathbb{D}(-)$.
\item If $X$ is a cubic fourfold or a GM fourfold and $Z$ is a closed subscheme of $X$, we denote~$I_Z$ to be the ideal sheaf of $Z$ in $X$. If $Z$ is also contained in a closed subscheme $Z'$ of $X$, then we denote $I_{Z/Z'}$ to be the ideal sheaf of $Z$ in $Z'$.
\item We use $\hom$ and $\ext^{i}$ to represent the dimension of the vector spaces $\Hom$ and~$\Ext^{i}$.
\end{itemize}


\subsection*{Acknowledgements}
 It is our pleasure to thank Arend Bayer and Qizheng Yin for very useful discussions on the topics of this project. We would like to thank Soheyla Feyzbakhsh, Yong Hu, Grzegorz Kapustka, Kieran O'Grady, Alexander Perry and Claire Voisin for helpful conversations. We are very grateful to Chunyi Li for reading the first draft of the article and providing many detailed suggestions for revisions. We thank Yilong Zhang for bringing the paper \cite{shinder2017geometry} to our attention. The second author would like to thank Huizhi Liu for encouragement and support. The third author is supported by the ERC Consolidator Grant WallCrossAG, no. 819864.


\section{Stability conditions on Kuznetsov components}\label{stability_condition_GM}
In this section, we recall (weak) Bridgeland stability conditions on a triangulated category and the notion of stability conditions on the Kuznetsov component of certain Fano fourfolds. We follow from \cite[Section 2]{bayer2017stability}.

\subsection{Weak stability conditions}

Let $\cD$ be a triangulated category and $K_0(\cD)$ be its Grothendieck group. Fix a surjective morphism to a finite rank lattice $v : K_0(\cD) \ra \Lambda$. 

\begin{definition}
The \emph{heart of a bounded t-structure} on $\cD$ is an abelian subcategory $\cA \sst \cD$ such that the following conditions are satisfied
\begin{enumerate}
    \item for any $E, F \in \cA$ and $n <0$, we have $\Hom(E, F[n])=0$,
    \item for any object $E \in \cD$, there exists a sequence of morphisms 
    \[ 0=E_0 \xrightarrow{\phi_1} E_1 \xrightarrow{\phi_2} \cdots \xra{\phi_m} E_m=E \]
    such that $\cone(\phi_i)$ is in the form $A_i[k_i]$, for some sequence $k_1 > k_2 > \cdots > k_m$ of integers and $A_i \in \cA$.
\end{enumerate}
\end{definition}

\begin{definition}
Let $\cA$ be an abelian category and $Z : K_0(\cA) \ra \mathbb{C}$ be a group homomorphism such that for any $E \in \cA$, we have $\Im Z(E) \geq 0$ and if $\Im Z(E) = 0$, $\Re Z(E) \leq 0$. Then we call~$Z$ a \emph{weak stability function} on $\cA$. Furthermore, for any $0\neq E\in\cA$, if we have $\Im Z(E) \geq 0$ and $\Im Z(E) = 0$ implies that $\Re Z(E) < 0$, then we call $Z$ a \emph{stability function} on $\cA$.
\end{definition}

\begin{definition}
A \emph{weak stability condition} on $\cD$ is a pair $\sigma = (\cA, Z)$, where $\cA$ is the heart of a bounded t-structure on $\cD$ and $Z : \Lambda \ra \CC$ is a group homomorphism such that 
\begin{enumerate}
    \item the composition $Z \circ v : K_0(\cA) \cong K_0(\cD) \ra \CC$ is a weak stability function on $\cA$. From now on, we write $Z(E)$ rather than $Z(v(E))$.
\end{enumerate}
Much like the slope from classical $\mu$-stability, we can define a \emph{slope} $\mu_\sigma$ for $\sigma$ using $Z$. For any object $E \in \cA$, set
\[
\mu_\sigma(E) := \begin{cases}  - \frac{\Re Z(E)}{\Im Z(E)}, & \Im Z(E) > 0 \\
+ \infty , & \text{else}.
\end{cases}
\]
We say an object $0 \neq E \in \cA$ is $\sigma$-(semi)stable if for any proper subobject $F \sst E$, we have $\mu_\sigma(F) < \mu_\sigma(E/F)$ (respectively $\mu_\sigma(F)\leq \mu_\sigma(E/F)$). 
\begin{enumerate}[resume]
    \item Any object $E \in \cA$ has a Harder--Narasimhan filtration in terms of $\sigma$-semistability defined above.
    \item There exists a quadratic form $Q$ on $\Lambda \otimes \mathbb{R}$ such that $Q|_{\ker Z}$ is negative definite  and $Q(E) \geq 0$ for all $\sigma$-semistable objects $E \in \cA$. This is known as the \emph{support property}.
\end{enumerate}
If the composition $Z \circ v$ is a stability function, then $\sigma$ is a \emph{stability condition} on $\cD$.
\end{definition}


\subsection{Stability conditions on the Kuznetsov components} In \cite{bayer2017stability}, the authors observe that a cubic fourfold $X$ admits a conic fibration such that the Kuznetsov component~$\Ku(X)$ can be embedded as an admissible subcategory in a derived category of modules over $\mathbb{P}^3$ with respect to the even part of the associated Clifford algebra $\mathcal{B}_0$. Then they construct a Bridgeland stability condition by restricting the weak stability condition from the twisted derived category~$\D^b(\mathbb{P}^3,\mathcal{B}_0)$ to $\Ku(X)$.
Inspired by this idea, in \cite{perry2019stability}, the authors embed the Kuznetsov component $\Ku(X)$ of a GM fourfold $X$ into a twisted derived category of modules over a quadric threefold, which is associated with a conic fibration of $X$. Then we get a family of stability conditions on $\Ku(X)$.

\begin{theorem}[{\cite[Theorem 1.2]{bayer2017stability}, \cite[Theorem 1.2]{perry2019stability}}]
\label{existence_stability_GM4}
Let $X$ be a cubic fourfold or a GM fourfold. Then the Kuznetsov component $\Ku(X)$ has a Bridgeland stability condition. 
\end{theorem}

\section{Proof of Theorem~\ref{main_theorem_1}}
In this section, we prove Theorem~\ref{main_theorem_1}. Firstly, we begin with two lemmas for both cubic fourfolds and GM fourfolds.


\begin{lemma} \label{tgt_inj}
Let $X$ be a cubic fourfold or a GM fourfold and $j:Y\hookrightarrow X$ be a smooth hyperplane section. Suppose $E\in \Ku(Y)$ is a simple object such that $\pr_X(j_*E)\in \Ku(X)$ is also simple, then the natural map induced by the functor $\pr_X\circ j_*$
\begin{equation} \label{diff}
d_{[E]}: \Ext^1(E, E)\to \Ext^1(\pr_X(j_*E), \pr_X(j_*E))
\end{equation}
is injective if and only if
\[\Hom(j^*\pr_X(j_*E), F)=k,\]
where $F\in \Ku(Y)$ is any non-trivial self-extension of $E$.
\end{lemma}

\begin{proof}
Any non-zero element $[F]$ in $\Ext^1(E, E)$ corresponds to a non-trivial extension 
\begin{equation} \label{ex_1}
 E\to F\to E.
\end{equation}
Then $d_{[E]}([F])\in \Ext^1(\pr_X(j_*E), \pr_X(j_*E))$ corresponds to a triangle
\begin{equation} \label{lem_seq}
    \pr_X(j_*E)\to \pr_X(j_*F)\to \pr_X(j_*E)
\end{equation}
obtained by applying $\pr_X\circ j_*$ to the extension \eqref{ex_1}. Thus $d_{[E]}$ is injective if and only if for $0\neq [F]\in \Ext^1(E, E)$, the triangle \eqref{lem_seq} is non-trivial. By assumption,
\[\Hom(\pr_X(j_*E), \pr_X(j_*E))=k,\]
we know that \eqref{lem_seq} is non-trivial if and only if $\Hom(\pr_X(j_*E), \pr_X(j_*F))=k$. Then the result follows from the adjunction of functors
\[\Hom(j^*\pr_X(j_*E), F)\cong\Hom(\pr_X(j_*E), j_*F)\cong\Hom(\pr_X(j_*E), \pr_X(j_*F)).\qedhere\]
\end{proof}

When $Y$ is a cubic threefold and $E\in \Ku(Y)$ is $\tau$-stable, then $\mathrm{Ext}^2(E,E)=0$ as in \cite[Lemma 5.9]{PY20}. Moreover, the moduli space of $\tau$-stable objects of the character $[E]$ is smooth of the expected dimension. 
When $Y$ is a GM threefold and $E\in \Ku(Y)$, then $\Ext^2(E,E)=0$ if and only if $[E]$ is a smooth point in the moduli space. Thus we have the following lemma.

\begin{lemma} \label{lag_lem}
Let $X$ be a cubic fourfold or a GM fourfold and $j:Y\hookrightarrow X$ be a smooth hyperplane section. Assume that the functor $\pr_X\circ j_*$ induces a rational map
\[r: \cM^Y_{\tau}(a,b) \dashrightarrow \cM^X_{\sigma}(a,b)\]
such that $\cM^Y_{\tau}(a,b)$ is birational to its image $L$. Furthermore, if $X$ is a GM fourfold, we assume that $\cM^Y_{\tau}(a,b)$ is generically smooth of the expected dimension.
Then $L$ is a Lagrangian subvariety.
\end{lemma}

\begin{proof}
By assumption, we can choose an open dense subset $U\subset\cM^Y_{\tau}(a,b)$ such that $\cM^Y_{\tau}(a,b)$ is smooth over $U$ and $L$ is smooth over $r(U)$. At the same time, 
$r|_U$ is an isomorphism to its image. Moreover, for every object $[E]\in U$, the tangent map~$dr_{[E]}$ is injective and $\Ext^2(E,E)=0$.

Note that $\dim \cM^Y_{\tau}(a,b)=\frac{1}{2}\dim \cM^X_{\sigma}(a,b)$. In order to show that $L$ is Lagrangian, it suffices to prove that the canonical two form of $\cM_{\sigma}^X(a,b)$ 
becomes zero after restricting to $L$. Since $r(U)$ is dense in $L$, this is equivalent to prove that the restricting two form vanishes at every point of~$r(U)$.

Indeed, for $[F]\in\cM_{\sigma}^X(a,b)$, $\mathrm{Ext}^1(F,F)$ is identified with the tangent space $T_{[F]}\cM_{\sigma}^X(a,b)$. Then the holomorphic symplectic two form of $\cM_{\sigma}^X(a,b)$ is naturally given by the Yoneda pairing
$$\mathrm{Ext}^1(F,F)\times\mathrm{Ext}^1(F,F)\xrightarrow{\phi }\mathrm{Ext}^2(F,F)\cong\mathrm{Hom}(F,F)\cong k.$$ 
This means that for any $x, y\in \mathrm{Ext}^1(F,F)$, $\phi(x,y): F\to F[2]$ is the composition map $x[1]\circ y$.

Now we assume that $[F]=r([E])$ for an object $[E]\in U$. Then the restriction of $\phi$ at the point~$[F]=r([E])=[\pr_X(j_*E)]$ is naturally given by the Yoneda pairing 
$$\mathrm{Ext}^1(E,E)\times\mathrm{Ext}^1(E,E)\xrightarrow{\phi'}\mathrm{Ext}^2(E,E),$$ 
where $\mathrm{Ext}^1(E,E)\cong d_{[E]}(\Ext^1(E, E))$ is identified with the tangent space $T_{[F]}L$.
Indeed, we know that the functor $\pr_X\circ j_*$ induces an embedding $d_{[E]}:\Ext^1(E,E)\hookrightarrow \Ext^1(F, F)$, mapping an element $a:E\to E[1]$ to $\pr_X\circ j_*(a): F\to F[1]$. 

Then for $a,b \in \Ext^1(E,E)$, let $x:=\pr_X\circ j_*(a)$ and $y:=\pr_X\circ j_*(b)$, $x, y\in d_{[E]}(\Ext^1(E, E))$. The restriction of $\phi$ on $d_{[E]}(\Ext^1(E, E))$ is given by $$\phi(x, y)=x[1]\circ y=\pr_X\circ j_*(a[1]\circ b)\in d_{[E]}(\Ext^2(E, E)).$$ 
Since $\Ext^2(E,E)=0$, $\phi'$ vanishes at every point of $r(U)$ and the desired result follows.
\end{proof}


\subsection{$(a,b)=(1,1)$: The Fano variety of lines}
By \cite[Theorem 1.1]{li2018twisted}, the Fano variety of lines $F(X)$ is isomorphic to the moduli space $\cM_{\sigma}^X(1,1)$. On the other hand, by \cite[Theorem 1.1]{PY20}, the Fano surface of lines $F(Y)$ is isomorphic to $\mathcal{M}_{\tau}^Y(1,1)$. 


Let $[E]$ be a point in $\cM_{\tau}^Y(1,1)$. By \cite[Proposition 4.6]{PY20}, we can assume that $E\cong I_{l/Y}$ for some line $l\subset Y$. Then according to \cite[Proposition 4.5(3)]{li2020elliptic},
$$P_l:=\mathrm{cone}(I_l[-1]\xrightarrow{ev}\oh_X(-H)[1]),\quad \mathrm{pr}_X(j_*E)\cong P_l.$$
Moreover, $P_l$ is $\sigma$-stable by \cite[Theorem 1.1]{li2018twisted}. Since $\mathcal{M}_{\tau}^Y(1,1)\cong F(Y)$, the moduli space~$\mathcal{M}_{\tau}^Y(1,1)$ admits a universal family. Then by the standard argument as in \cite[Theorem 3.9]{li2018twisted}, the Fourier--Mukai type functor $\pr_X\circ j_*$ induces a morphism $$\mathcal{M}_{\tau}^Y(1,1)\xrightarrow{f}\mathcal{M}_{\sigma}^X(1,1),\quad [I_{l/Y}]\mapsto[P_l].$$ Next we show that $f$~is an embedding, realizing $\mathcal{M}_{\tau}^Y(1,1)$ as a Lagrangian subvariety of $\mathcal{M}_{\sigma}^X(1,1)$.



\begin{theorem}
\label{Fano_variety_lines}
Let $X$ be a cubic fourfold and $j:Y\hookrightarrow X$ be a smooth hyperplane section.
Then the functor $\pr_X\circ j_*$ induces a closed embedding
\[f:\mathcal{M}_{\tau}^Y(1,1)\hookrightarrow\mathcal{M}_{\sigma}^X(1,1).\]
In particular, $\cM_{\tau}^Y(1,1)$ is a Lagrangian subvariety of $\cM_{\sigma}^X(1,1)$.
\end{theorem}

\begin{proof}
The object $\pr_X(j_*(I_{l/Y}))\cong P_l$ is determined by the line $l$, it is clear that $f$ is injective. Since both moduli spaces are proper, to show that $f$ is a closed embedding, we only need to show that the tangent map of $f$ is injective.

Let $[E]\in\mathcal{M}_{\tau}^Y(1,1)$, we assume that $E\cong I_{l/Y}$ for a line $l$ on $Y$. The tangent map of~$f$ is given by the map \eqref{diff}
\[d_{[E]}: \Ext^1(E, E)\to \Ext^1(\pr_X(j_*E), \pr_X(j_*E)).\]
By the definition of $\pr_X$, we have a triangle
\[\pr_X(j_*E)\to j_*E\to  \oh_X(-H)[1]\oplus \oh_X(-H)[2],\]
then we obtain a triangle
\begin{equation} \label{line_seq_1}
    j^*\pr_X(j_*E)\to j^*j_*E\to  \oh_Y(-H)[1]\oplus \oh_Y(-H)[2].
\end{equation}
Taking the long exact sequence of cohomology on  \eqref{line_seq_1}, we get
\[\oh_Y(-H)[1]\to j^*\pr_X(j_*E)\to \cH^0(j^*\pr_X(j_*E))\]
and an exact sequence
\[0\to \oh_l(-H)\to \cH^0(j^*\pr_X(j_*E))\to I_{l/Y}\to 0.\]
If $F$ is a non-trivial self-extension of $E\cong I_{l/Y}$, we have 
$$\Hom(j^*\pr_X(j_*E), F)\cong\Hom(\cH^0(j^*\pr_X(j_*E)), F)\cong\Hom(I_{l/Y}, F)=k.$$
By Lemma~\ref{tgt_inj}, the tangent map $df_{[E]}$ of $f$ at the point $[E]$ is injective. This shows that $f$ is a closed embedding. Furthermore, the subvariety $\cM_{\tau}^Y(1,1)$ is Lagrangian according to Lemma~\ref{lag_lem}.
\end{proof}


There is a natural embedding $i:F(Y)\hookrightarrow F(X)$, mapping $[l\subset Y]$ to $[l\subset X]$. From the construction above, it is not hard to see that the embedding $f$ is compatible with $i$.

\begin{corollary}\label{commupatible_Fano_var_of_lines}
The embedding in Theorem~\ref{Fano_variety_lines} is compatible with the natural one, which means that we have a commutative diagram
\[\begin{tikzcd}
	{F(Y)} & {\mathcal{M}_{\tau}^Y(1,1)} \\
	{F(X)} & {\mathcal{M}_{\sigma}^X(1,1)}
	\arrow["f", hook, from=1-2, to=2-2]
	\arrow["i"', hook, from=1-1, to=2-1]
	\arrow["\cong"', from=2-1, to=2-2]
	\arrow["\cong"', from=1-1, to=1-2]
\end{tikzcd}\]
\end{corollary}

\subsection{$(a,b)=(2,1)$: The LLSvS eightfold}
Let $X$ be a cubic fourfold not containing a plane and $M_3(X)$ be the irreducible component of the Hilbert scheme of twisted cubics on $X$. In \cite{LLSvS17}, the authors construct a two-step contraction on $M_3(X)$
$$\alpha: M_3(X)\xrightarrow{\alpha_1}Z'\xrightarrow{\alpha_2}Z,$$
where $\alpha_1:M_3(X)\rightarrow Z'$ is a $\mathbb{P}^2$-bundle and $\alpha_2:Z'\rightarrow Z$ is blowing up the image of the embedding $\mu: X\hookrightarrow Z$. Usually, we call $Z$ the \emph{LLSvS eightfold}. On the other hand, if we consider all twisted cubics contained in a hyperplane section $Y$ and denote its image under $\alpha$ by $Z_Y$, then it is shown that $Z_Y$ is Lagrangian in $Z$ by \cite{shinder2017geometry}. 

The LLSvS eightfold $Z$ is reconstructed as the moduli space $\cM_{\sigma}^X(2,1)$ in \cite{li2018twisted}. 
Now we consider the moduli space $\cM_{\tau}^Y(2,1)$, it is a smooth projective variety of dimension four and the geometry of this moduli space is intensively studied in \cite{bayer2020desingularization} and \cite{altavilla2019moduli}. First of all, we identify the moduli space $\cM_{\tau}^Y(2,1)$ with $Z_Y$. 

\begin{theorem}
\label{SS_variety_equal_BBF_moduli_space}
Let $Y$ be a general cubic threefold, then the moduli space $\cM_{\tau}^Y(2,1)\cong Z_Y$. 
\end{theorem}

We break the proof of Theorem~\ref{SS_variety_equal_BBF_moduli_space} into several lemmas. 

We define the projection functor  $\mathrm{pr}_Y:=\bL_{\oh_Y}\bL_{\oh_Y(H)}: D^b(Y)\rightarrow\Ku(Y)$ with respect to the semi-orthogonal decomposition
$$D^b(Y)=\langle\Ku(Y),\oh_Y, \oh_Y(H)\rangle.$$

Let $M_3(Y)$ be the irreducible component of the Hilbert scheme $\mathrm{Hilb}^{3t+1}(Y)$ containing smooth cubics on $Y$. We call curves in $M_3(Y)$ the (generalised) twisted cubics.

\begin{lemma}
\label{projection_object_cubics}
Let $Y$ be a smooth cubic threefold and $C\in M_3(Y)$. We denote the complex $$E_C:=\pr_Y(I_{C/Y}(2H))[-1].$$
\begin{enumerate}
    \item If $C$ is aCM, then $E_C\cong\ker(\oh_Y^{\oplus 3}\xra{ev} \oh_S(D))$ such that $S:=\langle C \rangle\cap Y$ is the cubic surface containing $C$ and $I_{C/S}(2H)\cong\oh_S(D)$, where $D$ is a Weil divisor on $S$. 
    \item If $C$ is not aCM, then 
    $E_C\cong\ker(\oh_Y^{\oplus 4}\xra{ev} I_{p/Y}(H))$, 
    where $p$ is the embedded point of $C$.
\end{enumerate}
\end{lemma}

\begin{proof}\leavevmode
By \cite[(1.2.2)]{LLMS18}, it is not hard to see $\bL_{\oh_Y(H)}I_{C/Y}(2H)$ sits in the triangle
$$\oh_Y(H)\rightarrow I_{C/Y}(2H)\rightarrow\bL_{\oh_Y(H)}I_{C/Y}(2H)\cong I_{C/S}(2H),$$
where $S:=\langle C \rangle\cap Y$.
Then $E_C\cong\bL_{\oh_Y}(I_{C/S}(2H))[-1]$. Note that $\RHom(\oh_Y, I_{C/S}(2H))=k^3[0]$ by \cite[(1.2.2)]{LLMS18}, we have a triangle
$$E_C\to \oh_Y^{\oplus 3}\xra{ev} I_{C/S}(2H).$$
Thus we obtain that $E_C\cong\ker(\oh_Y^{\oplus 3}\xra{ev} I_{C/S}(2H))$.

(1): If $C$ is aCM, by \cite[Proposition 3.1]{bayer2020desingularization}, we know that $S$ is normal and integral. Then by \cite[Proposition 3.2]{bayer2020desingularization}, if we set $D:=2H-C$, we have $I_{C/S}(2H)\cong \oh_S(D)$ as desired.

(2): If $C$ is not aCM, there are two short exact sequences
\begin{equation} \label{seq_kp}
    0\rightarrow I_{C/Y}(2H)\rightarrow I_{C_0/Y}(2H)\rightarrow k_p\rightarrow 0
\end{equation}
and 
\begin{equation}\label{SEQ_C0}
    0\rightarrow\oh_Y\rightarrow\oh_Y^{\oplus 2}(H)\rightarrow I_{C_0/Y}(2H)\rightarrow 0,
\end{equation}
where $C_0$ is a plane cubic curve and $p$ is the embedded point. The second exact sequence is the Koszul resolution of $I_{C_0/Y}$. Applying $\bL_{\oh_Y(H)}$ to (\ref{SEQ_C0}), we get $$\bL_{\oh_Y(H)}(I_{C_0/Y}(2H))\cong\oh_Y[1].$$
Then we obtain that $\mathrm{pr}_Y(I_{C_0/Y})\cong0$. 

Applying the projection functor $\mathrm{pr}_Y$ to (\ref{seq_kp}), we have $\mathrm{pr}_Y(k_p)[-1]\cong\mathrm{pr}_Y(I_{C/Y}(2H)).$ Then from $\bL_{\oh_Y(H)}k_p[-1]\cong I_{p/Y}(H)$, we know that $E_C$ is the kernel of the evaluation map
\[E_C\to \oh_Y^{\oplus 4}\xra{ev} I_{p/Y}(H).\qedhere\]
\end{proof}

In fact, $M_3(Y)$ admits a universal family. At the same time, $\mathrm{pr}_Y(I_{C/Y}(2H))$ is $\tau$-stable in~$\Ku(Y)$ for every Serre-invariant stability condition $\tau$ by Lemma \ref{projection_object_cubics} and \cite[Theorem 6.1(2), Theorem 8.7]{bayer2020desingularization}. Then according to the standard argument as in \cite[Theorem 3.9]{li2018twisted}, 
the functor $\pr_Y$ induces a dominant morphism $$\pi:M_3(Y)\rightarrow\cM_{\tau}^Y(2,1).$$ 
Indeed, $\pi$ is a proper surjective morphism. Note that the subvariety $Z_Y\subset Z$ is a two-step contraction of $M_3(Y)$ via the morphism $\alpha$. To show that $\cM_{\tau}^Y(2,1)$ is isomorphic to $Z_Y$, it is enough to prove that $\pi$ contracts the same locus as $\alpha$, which is equivalent to the statement in the following lemma.

\begin{lemma}
\label{projection_functor_coincide_alpha}
For twisted cubics $C$ and $C'$ on $Y$, $\mathrm{pr}_Y(I_{C/Y}(2H))\cong\mathrm{pr}_Y(I_{C'/Y}(2H))$ if and only if $\alpha(C)=\alpha(C')$, where $\alpha:M_3(Y)\rightarrow Z_Y$. Thus we have $Z_Y\cong \cM_{\tau}^Y(2,1)$. 
\end{lemma}

\begin{proof}
The argument is very similar to that in \cite[Proposition 2]{AL17}, but the situation here is simpler. 
If $\alpha(C)=\alpha(C')=p\in\mu(Y)$, then $C$ and $C'$ are both non-aCM twisted cubics with the embedded point $p$. Thus by Lemma~\ref{projection_object_cubics}(2), we have $\mathrm{pr}_Y(I_{C/Y}(2H))\cong\mathrm{pr}_Y(I_{C'/Y}(2H))$. If $\alpha(C)=\alpha(C')\notin\mu(Y)$, then $C$ and $C'$ are both aCM twisted cubics and they are in the same fiber of the $\mathbb{P}^2$-bundle map $\alpha_1$. This implies that they are in the same linear system, i.e., $I_{C/S}\cong I_{C'/S}$. Then $\mathrm{pr}_Y(I_{C/Y}(2H))\cong\mathrm{pr}_Y(I_{C'/Y}(2H))$.

Conversely, we show that if  $\pr_Y(I_{C/Y}(2H))\cong\pr_Y(I_{C'/Y}(2H))$, then $\alpha(C)=\alpha(C')$.
\begin{enumerate}
    \item If $C$ and $C'$ are both aCM, we need to show that $C$ and $C'$ are contained in the same cubic surface and in the same linear system. Let $S:=\langle C \rangle \cap Y$ and $S':=\langle C' \rangle \cap Y$. By assumption, we know that $E_C\cong E_{C'}$. Taking the same argument as in  \cite[Proposition 8.1]{GLZ2021conics}, we know that $\Hom(I_{C/S}, I_{C'/S'})\neq 0$. Then we have $I_{C/S}\cong I_{C'/S'}$ since they are Gieseker-stable. Hence $\alpha(C)=\alpha(C')$.
    
    \item If $C$ and $C'$ are not aCM with the embedded points $p$ and $p'$, by the isomorphism of projection objects  $\pr_Y(I_{C/Y}(2H))\cong\pr_Y(I_{C'/Y}(2H))$, we know that $E_C\cong E_{C'}$. Since $E_C$ is only non-locally free at $p$ and $E_{C'}$ is only non-locally free at $p'$, $p=p'$ and  $\alpha(C)=\alpha(C')$. 
    
    \item If $C$ is aCM and $C'$ is not aCM, then we prove that their projection objects in $\Ku(Y)$ can not be isomorphic. This is obvious since by Lemma~\ref{projection_object_cubics}, $E_C$ is locally free, while $E_{C'}$ is non-locally free at the embedded point of $C'$. Hence they cannot be isomorphic.\qedhere 
\end{enumerate}
\end{proof}

Next we show that the functor $\pr_X\circ j_*: \Ku(Y)\rightarrow\Ku(X)$ induces an embedding  $$p_j:\cM_{\tau}^Y(2,1)\to \cM_{\sigma}^X(2,1)$$ and $\cM_{\tau}^Y(2,1)$ is Lagrangian.

For a twisted cubic curve $C\subset X$ contained in a cubic surface $S\subset X$, let $F_C$ be the kernel of the evaluation map
\[ev: H^0(X, I_{C/S}(2H))\otimes \oh_X\twoheadrightarrow I_{C/S}(2H)\]
and $F'_C:=\bR_{\oh_X(-H)}\bL_{\oh_X}\bL_{\oh_Y(H)}F_C\cong\bR_{\oh_X(-H)}F_C\in\Ku(X)$ be the projection object of $F_C$ to the Kuznetsov component. By \cite[Lemma 2.3]{LLSvS17}, if $C$ is aCM, we have $F_C\cong F'_{C}$; if $C$ is not aCM, there is a non-splitting triangle
$$F'_C\rightarrow F_C\rightarrow\oh_X(-H)[1]\oplus\oh_X(-H)[2].$$
Thus for a non-aCM cubic $C$, we have $\cH^{-1}(F'_C)\cong\oh_X(-H)$ and a non-splitting exact sequence
\begin{equation} \label{H0FC}
    0\to \oh_X(-H)\to \cH^{0}(F'_C)\to F_C\to 0.
\end{equation}

\begin{proposition} \label{push-cubic}
Let $X$ be a cubic fourfold and $j: Y\hookrightarrow X$ be a smooth hyperplane section. Let $C$ be a twisted cubic on $Y$, then we have
\[\pr_X(j_*E_C)\cong F'_C.\]

\end{proposition}

\begin{proof}
Firstly, recall that for a twisted cubic $C\subset Y$, we have $\pr_X(j_*E_C)\cong\bR_{\oh_X(-H)}(j_*E_C)$. Then it suffices to prove that 
$\bR_{\oh_X(-H)}(j_*E_C)\cong  F'_C$. 
From Lemma \ref{projection_object_cubics}, we have a triangle
\[E_C\to \oh_Y^{\oplus 3}\xra{ev} I_{C/S}(2H).\]
Applying the functor $\bR_{\oh_X(-H)}\circ j_*$ to this triangle, we obtain
\begin{equation} \label{right-mut}
    \bR_{\oh_X(-H)}(j_*E_C)\to \bR_{\oh_X(-H)}(j_*\oh_Y^{\oplus 3})\to \bR_{\oh_X(-H)}(I_{C/S}(2H)).
\end{equation}
Note that $\bR_{\oh_X(-H)}(j_*\oh_Y)\cong \oh_X.$ Then we have 
\begin{equation} \label{right-mut_1}
    \bR_{\oh_X(-H)}(j_*E_C)\to \oh_X^{\oplus 3}\to \bR_{\oh_X(-H)}(I_{C/S}(2H)).
\end{equation}

Now we assume that $C$ is aCM.
We know that $\RHom(I_{C/S}(2H), \oh_X(-H))=0$. Then we have $\bR_{\oh_X(-H)}(I_{C/S}(2H))\cong I_{C/S}(2H)$. Thus the triangle (\ref{right-mut_1}) becomes
\[\bR_{\oh_X(-H)}(j_*E_C)\to \oh_X^{\oplus 3}\xra{ev} I_{C/S}(2H),\]
which implies that $\bR_{\oh_X(-H)}(j_*E_C)\cong F_C$ by the definition of $F_C$. Then we have $$\pr_X(j_*E_C)\cong F_C\cong F'_C.$$

Now we assume that $C$ is not aCM. We have $\RHom(\oh_X, I_{C/S})=k[-1]\oplus k[-2]$. Then by Serre duality, $\RHom(I_{C/S}(2H), \oh_X(-H))=k[-2]\oplus k[-3]$. Thus we get a non-splitting triangle
\[\bR_{\oh_X(-H)}(I_{C/S}(2H))\to I_{C/S}(2H)\to \oh_X(-H)[2]\oplus \oh_X(-H)[3].\]
Taking the long exact sequence of cohomology on \eqref{right-mut_1}, we obtain an isomorphism  $$\cH^{-1}(\bR_{\oh_X(-H)}(j_*E_C))\cong\oh_X(-H)$$ and an exact sequence
\[0\to \oh_X(-H)\to \cH^{0}(\bR_{\oh_X(-H)}(j_*E_C))\to \oh^{\oplus 3}_X\xra{ev} I_{C/S}(2H)\to 0.\]
Thus by definition of $F_C$, we have a non-splitting exact sequence
\[0\to \oh_X(-H)\to \cH^{0}(\bR_{\oh_X(-H)}(j_*E_C))\to F_C\to 0.\]
Note that $\Ext^1(F_C, \oh_X(-H))=k$, then we have $\cH^{0}(\bR_{\oh_X(-H)}(j_*E_C))\cong \cH^0(F'_C)$ by \eqref{H0FC}. Moreover, 
$\cH^{-1}(\bR_{\oh_X(-H)}(j_*E_C))\cong \cH^{-1}(F'_C)\cong \oh_X(-H)$. From the fact that 
\[\Hom(\cH^{0}(F'_C), \cH^{-1}(F'_C)[2])\cong \Hom(\cH^{0}(F'_C), \oh_X(-H)[2])=k,\]
such a non-trivial extension is unique up to an isomorphism. Hence we have $\pr_X(j_*E_C)\cong F'_C$.
\end{proof}

\begin{theorem} \label{8fold-embed}
Let $X$ be a cubic fourfold not containing a plane and  $j:Y\hookrightarrow X$ be a smooth cubic threefold. Then the functor $\pr_X\circ j_*$ induces a closed embedding 
\[p_j: \cM_{\tau}^Y(2,1)\hookrightarrow \cM_{\sigma}^X(2,1).\]
In particular, $\cM_{\tau}^Y(2,1)$ is a Lagrangian subvariety of $\cM_{\sigma}^X(2,1)$.
\end{theorem}

\begin{proof}
By \cite[Theorem 8.7]{bayer2020desingularization}, the moduli space $\mathcal{M}^Y_{\tau}(2,1)$ is isomorphic to the Gieseker moduli space $M_G(v)$ of stable sheaves of character $v=3-H-\frac{3}{2}L+\frac{1}{2}P$. Then by \cite[Theorem 4.6.5]{HL10}, the moduli space $M_G(v)$ is a fine moduli space. In other words, there is a universal family on $M_G(v)\times Y$. Then by the standard argument of \cite[Theorem 3.9]{li2018twisted} and Proposition~\ref{push-cubic}, the functor $pr_X\circ j_*$ induces a morphism $p_j:\mathcal{M}^Y_{\tau}(2,1)\rightarrow\mathcal{M}^X_{\sigma}(2,1)$. Next, we show that $p_j$ is injective. Indeed, if $p_j([E_C])=p_j([E_{C'}])$, then by Proposition~\ref{push-cubic}, we have $F_C\cong F_{C'}$. As in \cite[Theorem 3.9]{li2018twisted}, this implies that  $\alpha(C)=\alpha(C')$. Then according to  Lemma~\ref{projection_functor_coincide_alpha}, we obtain that  $E_C\cong E_{C'}$, i.e., $[E_C]=[E_{C'}]$. Thus $p_j$ is injective.

Since both moduli spaces are proper, to show that $p_j$ is a closed embedding, we only need to show that $p_j$ induces an injection on tangent spaces at closed points. By Lemma \ref{tgt_inj} and Proposition~\ref{push-cubic}, it suffices to check that $\Hom(j^*F'_C, F)=k$ for any non-trivial self-extension $F$ of $E_C$.

Now we assume that $C\subset Y$ is aCM. Then we have a triangle
\begin{equation}
    \bR_{\oh_X(-H)}(j_*E_C) \to j_*E_C \to \oh_X(-H)^{\oplus 3}[1].
\end{equation}
Since $j^*F'_C$ fits into the triangle
\[j^*F'_C\to \oh_Y^{\oplus 3}\to j^*I_{C/S}(2H)\]
and $j^*I_{C/S}(2H)$ fits into the triangle
\[I_{C/S}(H)[1]\to j^*I_{C/S}(2H)\to I_{C/S}(2H),\]
after taking the long exact sequence of cohomology, we know that $j^*F'_C$ is actually a sheaf on $Y$ and fits into an exact sequence
\[0\to I_{C/S}(H)\to j^*F'_C\to E_C\to 0.\]
Since $I_{C/S}(H)$ is torsion and $F$ is torsion free, we have $\Hom(j^*F'_C, F)\cong\Hom(E_C, F)=k$.  

Now we assume that $C\subset Y$ is non-aCM. By definition of $\bR_{\oh_X(-H)}j_*E_C$, we have a triangle
\[F'_C\cong\bR_{\oh_X(-H)}j_*E_C\to j_*E_C\to \oh_X(-H)^{\oplus 4}[1]\oplus \oh_X(-H)[2].\]
In this case, $\cH^{-1}(j^*F'_C)\cong \oh_Y(-H)$ and we have an exact sequence
\[0\to E_C(-H)\to \oh_Y^{\oplus 4}(-H)\to \cH^{0}(j^*F'_C)\to E_C\to 0.\]
Thus by Lemma \ref{projection_object_cubics}, we have an exact sequence
\begin{equation} \label{seq_12}
    0\to I_{p/Y}\to \cH^{0}(j^*F'_C)\to E_C\to 0.
\end{equation}
After applying $\Hom(-,F)$ to \eqref{seq_12}, we obtain that  \[\Hom(j^*F'_C, F)\cong\Hom(\cH^{0}(j^*F'_C), F)\cong\Hom(E_C, F)=k.\]
This is what we desired. Therefore, $p_j$ is a closed embedding. 

Finally, according to Lemma \ref{lag_lem}, we know that 
the subvariety $\cM_{\tau}^Y(2,1)$ is Lagrangian.
\end{proof}





In \cite{shinder2017geometry}, the authors show that the natural embedding $i':Z_Y\hookrightarrow Z$ realizes $Z_Y$ as a Lagrangian subvariety of $Z$. We can prove this fact by another method. Indeed, by Theorem~\ref{8fold-embed}, after the identification of $\cM_{\tau}^Y(2,1)$ with $Z_Y$ and $\cM_{\sigma}^X(2,1)$ with $Z$, we get an embedding $Z_Y\hookrightarrow Z$ which realizes $Z_Y$ as a Lagrangian subvariety of $Z$. Actually, the two embeddings $p_j$ and $i'$ are compatible.

\begin{corollary}\label{commupatible_LLSvS}
The embedding in Theorem \ref{8fold-embed} is compatible with the one in \cite{shinder2017geometry}, which means that we have a commutative diagram
\[\begin{tikzcd}
	{Z_Y} & {\mathcal{M}_{\tau}^Y(2,1)} \\
	{Z} & {\mathcal{M}_{\sigma}^X(2,1)}
	\arrow["{p_j}", hook, from=1-2, to=2-2]
	\arrow["{i'}"', hook, from=1-1, to=2-1]
	\arrow["\cong"', from=2-1, to=2-2]
	\arrow["\cong"', from=1-1, to=1-2]
\end{tikzcd}\]
\end{corollary}

\section{Proof of Theorem~\ref{main_theorem_2}}

In this section, we prove Theorem~\ref{main_theorem_2}. We start with a result in \cite{JLLZ2021gushelmukai}.

\begin{theorem}[{\cite{JLLZ2021gushelmukai}}]\label{mod_215_epwsurface}
Let $Y$ be a general ordinary GM threefold. Then we have
\[\mathcal{M}_{\tau'}^Y(1,0)\cong\widetilde{Y}^{\geq 2}_{A(Y)^{\perp}},\quad\mathcal{M}_{\tau'}^Y(0,1)\cong \widetilde{Y}^{\geq 2}_{A(Y)},\]
where $\widetilde{Y}^{\geq 2}_{A(Y)^{\perp}}$ is the double dual EPW surface and  $\widetilde{Y}^{\geq 2}_{A(Y)}$ is the double EPW surface associated with $Y$.
\end{theorem}

\begin{proof}
From \cite{Log12} and \cite{Debarre2021quadrics}, we know that $\widetilde{Y}^{\geq 2}_{A(Y)^{\perp}}\cong \cC_m(Y)$, where $\cC_m(Y)$ is the minimal model of the Fano surface of conics $F_g(Y)$ on $Y$. At the same time, $\widetilde{Y}^{\geq 2}_{A(Y)}\cong \cC_m(Y_L)$, where $Y_L$ is a line transform (period dual) of $Y$. Then the result follows from \cite[Theorem 7.13, Corollary 10.5]{JLLZ2021gushelmukai}.
\end{proof}

\subsection{$(a,b)=(1,0)$: The double dual EPW sextic}
Let $X$ be a very general ordinary GM fourfold and $Y$ be its smooth hyperplane section. By \cite[Theorem 1.1]{GLZ2021conics}, we have an isomorphism $\cM_{\sigma'}^X(1,0)\cong\widetilde{Y}_{A(X)^{\perp}}$ between the moduli space and the double dual EPW sextic of $X$, which is a hyperkähler fourfold. 

Now we aim to show that the functor $\pr_X\circ j_*:\Ku(Y)\to\Ku(X)$ induces an embedding $$q_j: \cM_{\tau'}^Y(1,0)\hookrightarrow\cM_{\sigma'}^X(1,0).$$ In particular, $\cM_{\tau'}^Y(1,0)$ is Lagrangian in $\cM_{\sigma'}^X(1,0)$, Moreover,  $q_j$ is compatible with the embedding $i'': \widetilde{Y}^{\geq 2}_{A(Y)^{\perp}}\hookrightarrow \widetilde{Y}_{A(X)^{\perp}}$ in \cite[Proposition 5.1]{iliev2011fano} via the identifications above. 

\begin{proposition}\label{pr2-pr}
Let $X$ be a very general ordinary GM fourfold and $j:Y\hookrightarrow X$ be a smooth hyperplane section. Let $C\subset Y$ be a  conic. Then we have
\[\pr_X(j_*\pr_Y(I_{C/Y}))\cong\pr_X(I_{C}),\]where $\pr_Y:=\bL_{\oh_Y}\bL_{\cU^{\vee}_Y}$ and $\mathrm{pr}_X:=\bR_{\cU}\bR_{\oh_X(-H)}\bL_{\oh_X}\bL_{\cU^{\vee}}$ are the projection functors to the Kuznetsov components.
\end{proposition}

\begin{proof}
There are three types of conics on an ordinary GM threefold: $\tau$-conics, $\rho$-conics and $\sigma$-conics. If $C$ is a $\tau$-conic or a $\rho$-conic, then $\pr_Y(I_{C/Y})\cong I_{C/Y}$ and the statement follows from \cite[Lemma 6.1]{GLZ2021conics}. If $C$ is a $\sigma$-conic, then by \cite[Proposition 7.2]{JLLZ2021gushelmukai}, we have a triangle in~$D^b(Y)$
\[\cU_Y[1]\to \pr_Y(I_{C/Y})\to \cQ^{\vee}_Y,\]
then we obtain a triangle in $\D^b(X)$
\[j_*\cU_Y[1]\to j_*\pr_Y(I_{C/Y})\to j_*\cQ^{\vee}_Y.\]
As in \cite[Proposition 6.8]{GLZ2021conics}, there are  isomorphisms $$\bR_{\oh_X(-H)}j_*\cU_Y\cong \cU\oplus \cQ(-H),\quad\bR_{\oh_X(-H)}j_*\cQ_Y^{\vee}\cong \cU\oplus \cQ^{\vee},$$ hence we get a triangle
\[\cU[1]\oplus \cQ(-H)[1]\to \bR_{\oh_X(-H)}(j_*\pr_Y(I_{C/Y}))\to \cU\oplus \cQ^{\vee}.\]
Finally, applying $\bR_{\cU}$ to this triangle, we obtain
\[\bR_{\cU}(\cQ(-H))[1]\to \pr_X(j_*\pr_Y(I_{C/Y}))\to \bR_{\cU}\cQ^{\vee}.\]
Since $\RHom(\cQ(-H), \cU)=k[0]$ and $\RHom(\cQ^{\vee}, \cU)=0$, we obtain the triangles
\[ \cQ(-H)\to \cU\to \bR_{\cU}(\cQ(-H))[1]\]
and
\[\bR_{\cU}(\cQ(-H))[1]\to \pr_X(j_*\pr_Y(I_{C/Y}))\to \cQ^{\vee}.\]

By \cite[Proposition 6.4]{GLZ2021conics}, there is a short exact sequence
$$0\rightarrow\cU\rightarrow\cQ^{\vee}\rightarrow I_q\rightarrow 0,$$
where $q$ is the unique $\sigma$-quadric surface in $X$. 
Thus we get 
$$\mathbb{D}(I_q(H))\rightarrow\cQ(-H)\rightarrow\cU.$$
Then $\mathrm{pr}_X(j_*\mathrm{pr}_Y(I_{C/Y}))$ is given by the triangle
$$\mathbb{D}(I_q(H))[1]\rightarrow\mathrm{pr}_X(j_*\mathrm{pr}_Y(I_{C/Y}))\rightarrow\cQ^{\vee}.$$
Note that $\Hom(\cQ^{\vee}, \mathbb{D}(I_q(H))[2])\cong\Hom(I_q(H), \cQ[2])=k$. Then the result follows from the triangle 
$\mathbb{D}(I_q(H))[1]\rightarrow\pr_X(I_{C})\rightarrow\cQ^{\vee}$
as in \cite[Proposition 6.6]{GLZ2021conics}. 
\end{proof}

Recall that the Serre functor of $\Ku(Y)$ is $\iota_Y[2]$, where $\iota_Y$ is an involution on $\Ku(Y)$. There is also an involution $T:=\bL_{\oh_X}\circ \mathbb{D}$ on $\Ku(X)$.

\begin{lemma} \label{invo_3fold}
Let $C$ be a conic on $Y$. Then we have
\[T(\pr_X(j_*\iota_Y\pr_Y(I_{C/Y})))\cong \pr_X(j_*\pr_Y(I_{C/Y})).\]
\end{lemma}

\begin{proof}
This follows from \cite[Proposition 7.3]{JLLZ2021gushelmukai} and \cite[Proposition 8.6]{GLZ2021conics}.
\end{proof}

Let $[E]$ be an object in the moduli space $\cM_{\tau'}^Y(1,0)$. By \cite[Theorem 7.11, Theorem 9.5]{JLLZ2021gushelmukai}, it is either in the form of $I_{C/Y}$ or given by the triangle  

\begin{equation} \label{sigma_conic_triangle_Y}
    \cU_Y[1]\rightarrow\mathrm{pr}_Y(I_{C/Y})\rightarrow\cQ^{\vee}_Y.
\end{equation}
Moreover, its image under the functor $\pr_X\circ j_*$ is also $\sigma'$-stable by Proposition~\ref{pr2-pr} and \cite[Theorem 7.1]{GLZ2021conics}. Note that the moduli space $\cM_{\tau'}^Y(1,0)$ admits a universal family (cf. \cite[Proposition 10.1]{JLLZ2021gushelmukai}). Then by the standard argument of \cite[Theorem 3.9]{li2018twisted}, the functor $\pr_X\circ j_*$ induces a morphism $q_j:\cM_{\tau'}^Y(1,0)\to\cM_{\sigma'}^X(1,0)$. 

To further prove that $q_j$ is an embedding, we begin with two lemmas.

\begin{lemma} \label{conic_not_intersect}
Let $C$ and $C'$ be two $\tau$-conics on $Y$. If  $\pr_X(I_C)\cong \pr_X(I_{C'})$, then the planes $\langle C \rangle $ and $\langle C' \rangle$ intersect at a single point.
\end{lemma}

\begin{proof}
By \cite[Proposition 8.1]{GLZ2021conics}, $C$ and $C'$ are contained in a same del Pezzo surface $\Sigma\subset X$. Furthermore, the generating planes $\langle C \rangle$ and $\langle C' \rangle$ belong to a singular quadric threefold $Q_{C, V_4}$, where $Q_{C, V_4}$ is a cone over a smooth quadratic surface $S'$. Then all planes in $Q_{C, V_4}$ correspond to  lines on $S'$, which has two rulings. By \cite[Proposition 4.9]{iliev2011fano} and \cite[Proposition 8.1]{GLZ2021conics}, the isomorphism $\pr_X(I_C)\cong \pr_X(I_{C'})$ implies that $\langle C \rangle$ and $\langle C' \rangle$ correspond to two disjoint lines in $S'$ but in the same ruling. Hence $\langle C \rangle $ and $\langle C' \rangle$ only intersect at the cone point of $Q_{C, V_4}$.
\end{proof}

\begin{lemma} \label{ext1=0}
Let $C$ be a $\tau$-conic on $Y$ and $D$ be the zero locus of a regular section of $\cU^{\vee}_Y$. Then we have $\Ext^1(I_{C/Y}, I_{D/Y})=0$.
\end{lemma}

\begin{proof}
From the Koszul resolution, we have an exact sequence 
\[0\to \oh_Y(-H)\to \cU_Y\to I_{D/Y}\to 0.\]
Then the result follows from applying $\Hom(I_{C/Y}, -)$ to this exact sequence, Serre duality and \cite[Lemma 6.2]{JLLZ2021gushelmukai}.
\end{proof}

\begin{theorem}\label{epw_surface_lag1}
Let $X$ be a very general ordinary GM fourfold and $j:Y\hookrightarrow X$ be a smooth hyperplane section. Then the functor $\pr_X \circ j_*$ induces a closed embedding
\[q_j: \cM_{\tau'}^Y(1,0)\hookrightarrow\cM_{\sigma'}^X(1,0),\]
realizing $\cM_{\tau'}^Y(1,0)$ as a Lagrangian subvariety of $\cM_{\sigma'}^X(1,0)$. 
\end{theorem}


\begin{proof}
By Proposition~\ref{pr2-pr} and the discussion above, we know that the functor $\pr_X\circ j_*$ induces a morphism $q_j:\cM_{\tau'}^Y(1,0)\to \cM_{\sigma'}^X(1,0)$. Firstly, we show that $q_j$ is injective. By the compatibility in Proposition \ref{pr2-pr}, $q_j$ maps the \text{Plücker} points, which are images under the contraction of $\sigma$-conics and the unique $\rho$-conic on $Y$, to the corresponding ones in $\cM_{\sigma'}^X(1,0)$. Then it is sufficient to show that $\pr_X\circ j_*$ is injective on the locus of the image under the projection of $\tau$-conics. Let~$C$ be a $\tau$-conic. Then by Proposition~\ref{pr2-pr}, $\pr_X(j_*I_{C/Y})\cong\pr_X(I_{C})$. For any two $\tau$-conics $C$ and~$C'$ on $Y$,  $\pr_X(j_*I_{C/Y})\cong\pr_X(j_*I_{C'/Y})$ implies that $\pr_X(I_C)\cong \pr_X(I_{C'})$. Furthermore, $C$ and $C'$ are both in a del Pezzo surface $\Sigma\subset X$ by \cite[Proposition 8.1]{GLZ2021conics}. Now let $D:=\Sigma\cap Y$, then~$D$ is a degree four elliptic curve and we have $D=C\cup C'$. But this implies that $\mathbb{P}^1\subset \langle C\rangle \cap \langle C'\rangle$, which contradicts Lemma \ref{conic_not_intersect}. Thus we know that $q_j$ is injective.



Now it is left to show that $q_j$ is injective at the level of tangent spaces. By Lemma \ref{tgt_inj}, we only need to prove that \[\Hom(j^*\pr_X(j_*\pr_Y(I_{C/Y})), F)=k,\]
for any non-trivial self-extension $F$ of $\pr_Y(I_{C/Y})$.

\textbf{Case 1: $\tau$-conic.} Let $C$ be a $\tau$-conic. In this case,  $$\pr_Y(I_{C/Y})\cong I_{C/Y},\quad \pr_X(j_*I_{C/Y})\cong \pr_X(I_{C}).$$ 
Then $F$ is a non-trivial self-extension of $I_{C/Y}$ and there are two natural maps $i_1:I_{C/Y}\hookrightarrow F$ and $i_2:F\twoheadrightarrow I_{C/Y}$. Recall that from \cite[Lemma 8.5]{GLZ2021conics}, we have an exact sequence
\[0\to I_{\Sigma}\to \pr_X(I_C)\to I_C\to 0,\]
where $\Sigma$ is the zero locus of a section of $\cU^{\vee}$ containing $C$. It is not hard to check that $j^*I_{\Sigma}\cong I_{D/Y}$, where $D$ is the zero locus of a section of $\cU^{\vee}_Y$. This means that $D$ is a degree four elliptic curve.
Then we have an exact sequence
\begin{equation}
    0\to I_{D/Y}\to j^*\pr_X(j_*I_{C/Y})\xra{s_1} j^*I_{C}\to 0
\end{equation}
and $j^*I_{C}$ fits into 
\[0\to \oh_C(-H)\to j^*I_C\xra{s_2} I_{C/Y}\to 0.\]

Since $\pr_X(j_*I_{C/Y})$ is $\sigma'$-stable, we know that 
$$\Hom(\pr_X(j_*I_{C/Y}), \pr_X(j_*I_{C/Y}))\cong\Hom(j^*\pr_X(j_*I_{C/Y}), I_{C/Y})=k.$$ 
Let $0\neq s_3\in \Hom(j^*\pr_X(j_*I_{C/Y}), I_{C/Y})=k$, then we have $s_3=s_2\circ s_1$. Taking the composition with $i_1$, we get a non-zero map \[s_4:=i_1\circ s_3=i_1\circ s_2\circ s_1:j^*\pr_X(j_*I_{C/Y})\to I_{C/Y}\hookrightarrow F.\]

In the following, we prove our statement by contradiction. If $\Hom(j^*\pr_X(j_*I_{C/Y}), F)\neq~k$, then we apply $\Hom(j^*\pr_X(j_*I_{C/Y}),-)$ to the exact sequence defining $F$, due to the fact that $\Hom(j^*\pr_X(j_*I_{C/Y}), I_{C/Y})=k$, we have $\Hom(j^*\pr_X(j_*I_{C/Y}), F)=k^2$. 

Let $s_5$ be a non-zero element in 
$\Hom(j^*\pr_X(j_*I_{C/Y}), F)=k^2$, linearly independent with $s_4$. Then $s_3=i_2\circ s_5$, since we have the exact sequence by our assumption 
\[0\to \Hom(j^*\pr_X(j_*I_{C/Y}), I_{C/Y})=\langle s_3\rangle\to \Hom(j^*\pr_X(j_*I_{C/Y}), F)=\langle s_4, s_5\rangle\] \[\to \Hom(j^*\pr_X(j_*I_{C/Y}), I_{C/Y})=\langle s_3\rangle\to 0.\]
Therefore, we have a commutative diagram

\[\begin{tikzcd}
	{j^*\pr_X(j_*I_{C/Y})} & {j^*I_C} \\
	F & {I_{C/Y}}
	\arrow["{i_2}"', two heads, from=2-1, to=2-2]
	\arrow["{s_5}"', from=1-1, to=2-1]
	\arrow["{s_1}"', two heads, from=1-1, to=1-2]
	\arrow["{s_2}"', two heads, from=1-2, to=2-2]
\end{tikzcd}\]
In addition, by the property of exact triangles, we obtain a commutative diagram

\[\begin{tikzcd}
	0 & {I_{D/Y}} & {j^*\pr_X(j_*I_{C/Y})} & {j^*I_C} & 0 \\
	0 & {I_{C/Y}} & F & {I_{C/Y}} & 0
	\arrow["{i_2}"', two heads, from=2-3, to=2-4]
	\arrow["{s_5}"', from=1-3, to=2-3]
	\arrow["{s_1}"', two heads, from=1-3, to=1-4]
	\arrow["{s_2}"', two heads, from=1-4, to=2-4]
	\arrow[from=1-4, to=1-5]
	\arrow[from=2-4, to=2-5]
	\arrow[from=1-2, to=1-3]
	\arrow[from=1-1, to=1-2]
	\arrow[from=2-1, to=2-2]
	\arrow["{i_1}"', from=2-2, to=2-3]
	\arrow["t"', from=1-2, to=2-2]
\end{tikzcd}\]
Note that $\ker(s_2)\cong\oh_C(-H)$. We claim that $t$ is a  non-zero map. Indeed, if $t$ is a zero map, then we have  $\mathrm{cok}(t)\cong I_{C/Y}$. Since $\Hom(\ker(s_2), \mathrm{cok}(t))=0$, by snake lemma we have an exact sequence
\[0\to \ker(t)=I_{D/Y}\to \ker(s_5)\to \oh_C(-H)\to 0.\]
This implies that $\ch(\Im(s_5))=\ch(I_{C/Y})$. Since $F$ is torsion free, we know that $\Im(s_5)\cong I_{C'/Y}$ for some conic $C'$ on $Y$. If $C'\neq C$, then from $\Hom(I_{C'/Y}, F)\cong\Hom(I_{C'/Y}, I_{C/Y})=0$, this is impossible. If $C'=C$, since $\Hom(I_{C/Y}, F)=\langle i_1\rangle$ and  $\Hom(j^*\pr_X(j_*I_{C/Y}), I_{C/Y})=\langle s_3\rangle$,  
$s_5$ is proportional to $s_4$, which is also impossible. Thus we know that $t$ is non-zero and must be the natural embedding $I_{D/Y}\hookrightarrow I_{C/Y}$.

Now we have a long exact sequence
\[0\to \ker(s_5)\to \oh_C(-H)\to I_{D/Y}\to \mathrm{cok}(s_5)\to 0.\]
Here $D=C\cup C'$, $I_{D/Y}\cong \oh_{C'}(-H)$ and $\Hom(\oh_C, \oh_{C'})=0$. Then $\ker(s_5)\cong\oh_C(-H)\cong\ker(s_2)$. Taking the quotient of the kernel, we have $G:=\Im(s_5)\subset F$, where $G$ fits into the exact sequence
\[0\to I_{D/Y}\to G\to I_{C/Y}\to 0.\]
From Lemma \ref{ext1=0}, we know that $G\cong I_{D/Y}\oplus I_{C/Y}\subset F$. In fact, this is impossible. On one hand, $i_3: I_{D/Y}\hookrightarrow G\to F$ is the composition of $I_{D/Y}\hookrightarrow j^*\pr_X(j_*I_{C/Y})\twoheadrightarrow G\hookrightarrow F$. On the other hand, from the commutative diagram above, this is just the map $i_1\circ t: I_{D/Y}\xra{t} I_{C/Y}\xra{i_1} F$. Then the image of the composition $I_{D/Y}\xra{\mathrm{id}\oplus 0} I_{D/Y}\oplus I_{C/Y}\hookrightarrow F$ is contained in $I_{C/Y}$. Thus we have $\mathrm{hom}(I_{C/Y},F)\geq2$, which contradicts the fact that $\Hom(I_{C/Y}, F)=k$. 

\textbf{Case 2: $\sigma$-conic.} Let $C$ be a $\sigma$-conic. Recall that we have a triangle
\[\mathbb{D}(I_q(H))[1]\to \pr_X(j_*\pr_Y(I_{C/Y}))\to \cQ^{\vee}.\]
Thus we get a triangle in $\D^b(Y)$
\[\mathbb{D}(I_{C'/Y}(H))[1]\to j^*\pr_X(j_*\pr_Y(I_{C/Y}))\to \cQ_Y^{\vee},\]
where $C'=q\cap Y$ is a $\sigma$-conic on $Y$. 

Recall that we have a triangle \eqref{sigma_conic_triangle_Y}
\[\cU_Y[1]\xra{v_1}\mathrm{pr}_Y(I_{C/Y})\xra{v_2}\cQ^{\vee}_Y.\]
Since $\pr_Y(I_{C/Y})\in \Ku(Y)$, we know that $\RHom(\cQ^{\vee}_Y, \pr_Y(I_{C/Y}))=0$. Therefore, according to \eqref{sigma_conic_triangle_Y}, we have $\Ext^1(\pr_Y(I_{C/Y}), \pr_Y(I_{C/Y}))\cong\Hom(\cU_Y[1], \pr_Y(I_{C/Y})[1])$. Then for any non-trivial map $v_3:\pr_Y(I_{C/Y})\to \pr_Y(I_{C/Y})[1]\in \Ext^1(\pr_Y(I_{C/Y}), \pr_Y(I_{C/Y}))$, the composition 
\[v_3\circ v_1: \cU_Y[1]\xra{v_1} \pr_Y(I_{C/Y})\xra{v_3} \pr_Y(I_{C/Y})[1]\]
is a non-zero element in $\Hom(\cU_Y[1], \pr_Y(I_{C/Y})[1])$. Applying $\Hom(\cU_Y[1],-)$ to the triangle \eqref{sigma_conic_triangle_Y}, we obtain $\Hom(\cU_Y[1], \pr_Y(I_{C/Y})[1])\cong\Hom(\cU_Y[1], \cQ^{\vee}_Y[1])$. Thus the composition
\[v_2[1]\circ v_3\circ v_1: \cU_Y[1]\to \cQ^{\vee}_Y[1]\]
is non-trivial.

Let $F$ be a non-trivial self-extension of $\pr_Y(I_{C/Y})$
\[\pr_Y(I_{C/Y})\to F\to \pr_Y(I_{C/Y}),\]
corresponding to the element $v_3:\pr_Y(I_{C/Y})\to \pr_Y(I_{C/Y})[1]\in \Ext^1(\pr_Y(I_{C/Y}), \pr_Y(I_{C/Y}))$. Taking the long exact sequence of cohomology, we get a long exact sequence
\[0\to \cU_Y\to \cH^{-1}(F)\to \cU_Y\xra{(v_2[1]\circ v_3\circ v_1)[-1]} \cQ^{\vee}_Y\to \cH^0(F)\to \cQ^{\vee}_Y\to 0.\]
Since the map $(v_2[1]\circ v_3\circ v_1)[-1]$ is non-trivial, it is not hard to see that $\cH^{-1}(F)\cong \cU_Y$ and we have an exact sequence
\[0\to I_{C''/Y}\to \cH^0(F)\to \cQ^{\vee}_Y\to 0,\]
where $C''$ is a $\sigma$-conic on $Y$. Since $\RHom(\cQ^{\vee}_Y, \pr_Y(I_{C/Y}))=0$, we know that \[\Hom(j^*\pr_X(j_*\pr_Y(I_{C/Y})), F)\cong\Hom(\mathbb{D}(I_{C'/Y}(H)), F).\]
Note that $\Hom(\mathbb{D}(I_{C'/Y}(H))[1], F)\cong\Hom(\mathbb{D}(F)[1], I_{C'/Y}(H))$ and there is a triangle
\begin{equation} \label{DF_seq}
    \mathbb{D}(\cH^0(F))[1] \to \mathbb{D}(F)[1]\to \cU^{\vee}_Y.
\end{equation}
The cohomology objects of $\mathbb{D}(\cH^0(F))[1]$ are  $\cH^0(\mathbb{D}(\cH^0(F))[1])\cong \oh_{C''}$ and 
\[0\to \cQ_Y\to \cH^{-1}(\mathbb{D}(\cH^0(F))[1])\to \oh_Y\to 0.\]
Then it is not hard to see that $\Hom(\mathbb{D}(\cH^0(F))[1], I_{C'/Y}(H)[i])=0$ for all $i\leq 0$. Applying $\Hom(-, I_{C'/Y}(H))$ to the triangle \eqref{DF_seq}, we obtain $\Hom(\mathbb{D}(F)[1], I_{C'/Y}(H))\cong\Hom(\cU^{\vee}_Y, I_{C'/Y}(H))$. 
Since $C'$ is a $\sigma$-conic, 
then the fact $\Hom(j^*\pr_X(j_*\pr_Y(I_{C/Y})),F)\cong\Hom(\cU^{\vee}_Y, I_{C'/Y}(H))=k$ follows from \cite[Lemma 6.2]{JLLZ2021gushelmukai}.

\textbf{Case 3: $\rho$-conic.} Let $C$ be a $\rho$-conic. Then by Lemma \ref{invo_3fold}, the map induced by $\pr_X\circ j_*$
\[d_1: \Ext^1(I_{C/Y}, I_{C/Y})\to \Ext^1(\pr_X(j_*I_{C/Y}), \pr_X(j_*I_{C/Y}))\]
is naturally isomorphic to the map induced by $T\circ \pr_X \circ j_*\circ \iota_Y$
\[d_2: \Ext^1(I_{C/Y}, I_{C/Y})\to \Ext^1(T(\pr_X(j_*\iota_Y I_{C/Y})), T(\pr_X(j_*\iota_Y I_{C/Y}))).\]
By \cite[Proposition 7.3]{JLLZ2021gushelmukai}, $\iota_Y I_{C/Y}\cong \pr_X(I_{C'/Y})$, where $C'$ is a $\sigma$-conic on $Y$. Combined with the result in Case 2, we know that the map $d_3$ induced by $\pr_X \circ j_*$
\[d_3: \Ext^1(\iota_Y I_{C/Y}, \iota_Y I_{C/Y})\to \Ext^1(\pr_X(j_*\iota_Y I_{C/Y}), \pr_X(j_*\iota_Y I_{C/Y}))\]
is injective. Then the result that $d_1$ is injective simply follows from the fact that $\iota_Y$ and $T$ is an auto-equivalence on $\Ku(Y)$ and $\Ku(X)$ respectively.


Finally, we have established the proof that $q_j$ is injective at the level of tangent spaces. For a general GM threefold $Y$, analogous to the case of cubic threefolds, 
the moduli space $\cM_{\tau'}^Y(1,0)$ is smooth and for any object $[E]$, $\mathrm{Ext}^2(E,E)=0$. Then the result that $\cM_{\tau'}^Y(1,0)$ is a Lagrangian subvariety of $\cM_{\sigma'}^X(1,0)$ follows from Lemma \ref{lag_lem}.
\end{proof}



\begin{corollary}\label{compatible_epw}
The embedding in Theorem~\ref{epw_surface_lag1} is compatible with the one in \cite[Lemma 5.1]{iliev2011fano}, which means that there is a commutative diagram
\[\begin{tikzcd}
	{\widetilde{Y}^{\geq 2}_{A^{\perp}}} & {\mathcal{M}_{\tau'}^Y(1,0)} \\
	{\widetilde{Y}_{A^{\perp}}} & {\mathcal{M}_{\sigma'}^X(1,0)}
	\arrow["{q_j}", hook, from=1-2, to=2-2]
	\arrow["{i''}"', hook, from=1-1, to=2-1]
	\arrow["\cong"', from=2-1, to=2-2]
	\arrow["\cong"', from=1-1, to=1-2]
\end{tikzcd}\]
\end{corollary}

\subsection{Lagrangian subvarieties of the double EPW sextics}
Let $X$ be a very general ordinary GM fourfold such that the associated Lagrangian data is $A(X)$. By \cite[Theorem 1.1]{GLZ2021conics}, the moduli space $\mathcal{M}_{\sigma'}^X(0,1)\cong\widetilde{Y}_{A}$, which is the double EPW sextic. Let $X'$ be the period dual of~$X$ such that $A(X')=A^{\perp}$. Then there is a morphism  $F_g(X')\rightarrow\widetilde{Y}_{A}$ as in \cite{iliev2011fano}. Let $Y'$ be a general hyperplane section of $X'$. Replacing $X$ by its period dual $X'$ and rewriting the argument in Theorem~\ref{epw_surface_lag1}, it is straightforward to see that the moduli space $\mathcal{M}_{\tau'}^Y(0,1)$ is a Lagrangian subvariety of $\mathcal{M}_{\sigma'}^X(0,1)$.

\section{Lagrangian covering families and Possible examples for Conjecture~\ref{main_conjecture_2}}
\subsection{Lagrangian covering families}
In this section, we provide some examples confirming Conjecture~\ref{Conjecture_Lag_covering} as a corollary of Theorem~\ref{main_theorem_1} and \ref{main_theorem_2}.
In these cases, the Bridgeland moduli spaces on the Kuznetsov components of Fano fourfolds can be covered by the Lagrangian subvarieties constructed as the Bridgeland moduli spaces on the Kuznetsov components of the corresponding  hyperplane sections, i.e., Fano threefolds. 

\begin{corollary}\label{covering_family_examples}
\leavevmode\begin{enumerate}
    \item Let $X$ be a general cubic fourfold and $j:Y\hookrightarrow{X}$ be its hyperplane section. Then the moduli spaces $\mathcal{M}_{\sigma}^X(1,1)$ and $\mathcal{M}_{\sigma}^X(2,1)$ admit a Lagrangian covering family. 
    \item Let $X$ be a very general GM fourfold and $j:Y\hookrightarrow{X}$ be its hyperplane section. Then the moduli spaces $\mathcal{M}^X_{\sigma'}(1,0)$ and $\mathcal{M}^X_{\sigma'}(0,1)$ admit a Lagrangian covering family. 
\end{enumerate}
\end{corollary}
\begin{proof}
In all cases, for a general rational curve on $X$, it belongs to a hyperplane section. Then the results follow from Lemma~\ref{covering_lemma}, Theorem~\ref{main_theorem_1} and \ref{main_theorem_2}.
\end{proof}

\begin{lemma}\label{covering_lemma}
Let $X$ be a cubic fourfold or a GM fourfold and $j:Y\hookrightarrow X$ be any smooth hyperplane section. Let $M(X)$ be a component of the Hilbert scheme $\mathrm{Hilb}^{dt+1}(X)$ and $M(Y)$ be its restriction on $Y$. Assume that
\begin{enumerate}

\item there is a commutative diagram 
\[\begin{tikzcd}
	{M(Y)} & {\cM^Y_{\tau}(a, b)} & {} \\
	{M(X)} & {\cM^{X}_{\sigma}(a,b)}
	\arrow["{}", hook, from=1-2, to=2-2]
	\arrow["{}", two heads, from=2-1, to=2-2]
	\arrow["{}", hook, from=1-1, to=2-1]
	\arrow["{}", two heads, from=1-1, to=1-2]
\end{tikzcd}\]
where the first column is the natural embedding $M(Y)\hookrightarrow M(X)$ and the second column is a Lagrangian embedding 
\[\cM^Y_{\tau}(a, b)\hookrightarrow \cM^{X}_{\sigma}(a,b)\]
induced by the functor $\pr_X\circ j_*$.
At the same time, the row in the diagram is a surjective morphism induced by the functors $\pr_Y$ and $\pr_X$ respectively.

\item A general curve $C\in M(X)$ lies on a smooth hyperplane section.
\end{enumerate}

Then the moduli space $\cM^{X}_{\sigma}(a,b)$ admits a Lagrangian covering family.

\end{lemma}

\begin{proof}

Let $\mathbb{P}^{\vee}$ be the dual projective space, parametrizing the hyperplane sections of $X$. We define a subvariety
\[\cU:=\{(H, C)|H\in \mathbb{P}^{\vee}, C\subset X_H\}\subset \mathbb{P}^{\vee}\times M(X).\]
Let $\cU(a,b)$ be the image of $\cU$ under the morphism $\mathrm{id}\times \pi_X: \mathbb{P}^{\vee}\times M(X)\to \mathbb{P}^{\vee}\times \cM^{X}_{\sigma}(a,b)$. By the commutative diagram as in (1), the fibre of $\pr_1:\cU(a,b)\to \mathbb{P}^{\vee}$ over a smooth hyperplane section $Y$ is a Lagrangian subvariety $\cM^{Y}_{\tau}(a, b)$ of $\cM^{X}_{\sigma}(a,b)$.

Thus to prove $\cU(a,b)$ is a Lagrangian covering family of $\cM^{X}_{\sigma}(a,b)$, we only need to show that the morphism $\pr_2:\cU(a,b) \to \cM^{X}_{\sigma}(a,b)$ is surjective. Since both varieties are projective, we only need to show that $\pr_2$ is dominant. This follows from $(2)$ and the commutative diagram in (1).
\end{proof}

\begin{remark}
\label{voisin_map}
In a private communication with O'Grady, he suggests that one should be able to construct a Lagrangian covering family of the LLSvS eightfold via the Voisin map constructed in \cite{Voi16}
$$\nu: F(X)\times F(X)\dashrightarrow Z.$$ 
Indeed, considering the induced map $\nu_Y$ on a hyperplane section $Y$, $$\nu_Y:F(Y)\times F(Y)\dashrightarrow Z,$$
the closure of the image is exactly the Lagrangian subvariety $Z_Y$ of $Z$. On the other hand, the map $\nu_Y$ has a moduli theoretic interpretation. The Fano variety of lines $F(Y)$ is isomorphic to the moduli space $\mathcal{M}_{\tau}(\Ku(Y),\lambda_1+\lambda_2)$. The latter is also isomorphic to $\mathcal{M}_{\tau}(\Ku(Y),\lambda_1)$ through the  rotation functor $\bR:=\bL_{\oh_Y}(-\otimes\oh_Y(H))[1]$ because the functor preserves the stability condition as in \cite[Proposition 5.4]{PY20}. By Lemma~\ref{SS_variety_equal_BBF_moduli_space}, $Z_Y\cong\mathcal{M}_{\tau}(\Ku(Y),2\lambda_1+\lambda_2)$. Thus  the map $\nu_Y$ can be realized as families of extensions (cf. \cite{chen21}). It is worth mentioning that in \cite{voisin2018triangle}, the author uses the map $\nu$ to show that the LLSvS eightfold $Z$ admits a surface decomposition (cf. \cite[Definition 0.2]{voisin2018triangle}). Then by \cite[Proposition 0.11]{voisin2018triangle}, the LLSvS eightfold admits mobile algebraically coisotropic subvarieties of any codimension $n\leq 4$. 
\end{remark}


Now we present some possible examples confirming Conjecture~\ref{main_conjecture_2} and give a comment on moduli spaces of a non-primitive class. We study these examples in the sequel \cite{FLZlag2022} of the current paper.

\subsection{Lagrangian subvarieties of the double EPW cube}
Let $X$ be a very general GM fourfold. In a very recent paper \cite{kapustka2022epw}, the authors show that the moduli space $\cM_{\sigma'}^X(1,-1)$ is isomorphic to the double EPW cube $\widetilde{C}_A$ studied in \cite{IKKR19}. Let $Y$ be a general hyperplane section of $X$, which is a GM threefold. By a similar argument as in \cite[Proposition 6.2]{zhang2020bridgeland}, one could show that the Bridgeland moduli space $\mathcal{M}_{\tau'}^Y(1,-1)$ is a smooth threefold constructed as a divisorial contraction of the Hilbert scheme of twisted cubics $M_3(Y)$. This suggests that the moduli space $\cM_{\sigma'}^X(1,-1)$ may be the MRC quotient of $M_3(X)$. We expect that  $\mathcal{M}_{\tau'}^Y(1,-1)$ can be realized as a Lagrangian subvariety of $\mathcal{M}_{\tau'}^Y(1,-1)$ via the functor $\pr_X\circ j_*$. 

\subsection{Lagrangian subvarieties of a twelve-dimensional hyperkähler variety}
Let $X$ be a GM fourfold, denote by $k_x$ the skyscraper sheaf of a point $x\in X$. The character of $\mathrm{pr}_X(k_x)$ is $2\Lambda'_2+\Lambda'_1$. By \cite[Theorem 1.5]{perry2019stability}, the moduli space $\cM_{\sigma'}^X(1,2)$ is a smooth projective hyperkähler variety of dimension $12$. On the other hand, we define $\pr':=\bL_{\cU_Y}\bL_{\oh_Y}$ with respect to the semi-orthogonal decomposition 
$D^b(Y)=\langle\Ku(Y)',\cU_Y,\oh_Y\rangle$. The moduli space of stable objects related to a skyscraper sheaf $[\mathrm{pr}'(k_y)]$ is studied in \cite{jLz2021brillnoether}, which is a six dimensional smooth projective variety. It is easy to check that the character of  $\mathrm{ch}(\mathrm{pr}_X(\mathbb{D}(j_*v)))$ is  $\Lambda'_1+2\Lambda'_2$.
We expect that this moduli space over $\Ku(Y)'$ is a Lagrangian subvariety of  $\cM_{\sigma'}^X(1,2)$.


\subsection{Lagrangian subvarieties of O'Grady 10}

Although we only consider moduli spaces for primitive classes in Conjecture \ref{main_conjecture_1} and \ref{main_conjecture_2}, we expect that analogous results also hold for non-primitive classes.

Let $X$ be a cubic fourfold, consider the moduli space $\cM^{X, ss}_{\sigma}(2,2)$ of S-equivalence classes of semistable objects with character $2\Lambda_1+2\Lambda_2$. According to the computation in \cite{li2020elliptic}, it is not hard to see that $\pr_X\circ j_*$ induces an embedding $M^{inst}(Y)\hookrightarrow \cM^{X}_{\sigma}(2,2)$, where $M^{inst}(Y)$ is the moduli space of stable instanton sheaves. Furthermore, the moduli space $M^{inst}(Y)$ is isomorphic to $\cM^Y_{\tau}(2,2)$ (cf. \cite[Theorem 7.6]{liu2021note}). Although $\cM^{X,ss}_{\sigma}(2,2)$ is singular along the strictly semistable locus, it admits a symplectic resolution $\widetilde{M}$ (cf. \cite[Theorem 1.1]{li2020elliptic}). Moreover, the closure of $M^{inst}(Y)$ in $\widetilde{M}$ is a Lagrangian subvariety and appears as a fibre of the Lagrangian fibration of $\widetilde{M}$.

\bibliography{lagrangian}

\newcommand{\etalchar}[1]{$^{#1}$}
\begin{thebibliography}{BLM{\etalchar{+}}21}

\bibitem[AL17]{AL17}
Nicolas Addington and Manfred Lehn.
\newblock On the symplectic eightfold associated to a {P}faffian cubic
  fourfold.
\newblock {\em J. Reine Angew. Math.}, 731:129--137, 2017.

\bibitem[APR19]{altavilla2019moduli}
Matteo Altavilla, Marin Petkovic, and Franco Rota.
\newblock Moduli spaces on the {K}uznetsov component of {F}ano threefolds of
  index 2.
\newblock {\em arXiv preprint arXiv:1908.10986}, 2019.

\bibitem[Bai22]{Bai22}
Chenyu Bai.
\newblock On {A}bel--{J}acobi {M}aps of {L}agrangian {F}amilies.
\newblock {\em arXiv preprint arXiv:2203.06242}, 2022.

\bibitem[BBF{\etalchar{+}}20]{bayer2020desingularization}
Arend Bayer, Sjoerd Beentjes, Soheyla Feyzbakhsh, Georg Hein, Diletta
  Martinelli, Fatemeh Rezaee, and Benjamin Schmidt.
\newblock {The desingularization of the theta divisor of a cubic threefold as a
  moduli space}.
\newblock {\em To appear in Geom. Topol., arXiv:2011.12240}, 2020.

\bibitem[BLM{\etalchar{+}}21]{BLMNPS21}
Arend Bayer, Mart\'{\i} Lahoz, Emanuele Macr\`\i, Howard Nuer, Alexander Perry,
  and Paolo Stellari.
\newblock Stability conditions in families.
\newblock {\em Publ. Math. Inst. Hautes \'{E}tudes Sci.}, 133:157--325, 2021.

\bibitem[BLMS17]{bayer2017stability}
Arend Bayer, Mart{\'\i} Lahoz, Emanuele Macr{\`\i}, and Paolo Stellari.
\newblock {Stability conditions on Kuznetsov components}.
\newblock {\em (Appendix joint with Xiaolei Zhao) To appear in Ann. Sci.
  {\'E}c. Norm. Sup{\'e}r., arXiv:1703.10839}, 2017.

\bibitem[Che21]{chen21}
Huachen Chen.
\newblock The {V}oisin map via families of extensions.
\newblock {\em Math. Z.}, 299(3-4):1987--2003, 2021.

\bibitem[DK22]{Debarre2021quadrics}
Oliver Debarre and Alexander Kuznetsov.
\newblock {Gushel--Mukai varieties: quadrics}.
\newblock {\em In preparation}, 2022.

\bibitem[FLZ22]{FLZlag2022}
Soheyla Feyzbakhsh, Zhiyu Liu, and Shizhuo Zhang.
\newblock {A moduli theoretic approach to Lagrangian subvarieties of
  hyperk{\"a}hler varieties: General situation}.
\newblock {\em In preparation}, 2022.

\bibitem[FP21]{FeyzbakhshPertusi2021stab}
Soheyla Feyzbakhsh and Laura Pertusi.
\newblock {Serre-invariant Stability conditions and Ulrich bundles on cubic
  threefolds}.
\newblock {\em arXiv preprint arXiv:2109.13549}, 2021.

\bibitem[GLZ22]{GLZ2021conics}
Hanfei Guo, Zhiyu Liu, and Shizhuo Zhang.
\newblock {Conics on Gushel--Mukai fourfolds, EPW sextics and Bridgeland moduli
  spaces}.
\newblock {\em arXiv preprint arXiv:2203.05442}, 2022.

\bibitem[HL10]{HL10}
Daniel Huybrechts and Manfred Lehn.
\newblock {\em The geometry of moduli spaces of sheaves}.
\newblock Cambridge Mathematical Library. Cambridge University Press,
  Cambridge, second edition, 2010.

\bibitem[IKKR19]{IKKR19}
Atanas Iliev, Grzegorz Kapustka, Micha{\l} Kapustka, and Kristian Ranestad.
\newblock E{PW} cubes.
\newblock {\em J. Reine Angew. Math.}, 748:241--268, 2019.

\bibitem[IM11]{iliev2011fano}
Atanas Iliev and Laurent Manivel.
\newblock {Fano manifolds of degree ten and EPW sextics}.
\newblock {\em Ann. Sci. {\'E}c. Norm. Sup{\'e}r.}, 44(3):393--426, 2011.

\bibitem[JLLZ21]{JLLZ2021gushelmukai}
Augustinas Jacovskis, Xun Lin, Zhiyu Liu, and Shizhuo Zhang.
\newblock {Categorical Torelli theorems for Gushel--Mukai threefolds}.
\newblock {\em arXiv preprint arXiv:2108.02946}, 2021.

\bibitem[JLZ21]{jLz2021brillnoether}
Augustinas Jacovskis, Zhiyu Liu, and Shizhuo Zhang.
\newblock Brill--{N}oether theory and categorical {T}orelli theorems for
  {K}uznetsov components of index 1 {F}ano threefolds.
\newblock {\em preprint}, 2021.

\bibitem[KKM22]{kapustka2022epw}
Grzegorz Kapustka, Micha{\l} Kapustka, and Giovanni Mongardi.
\newblock {EPW sextics vs EPW cubes}.
\newblock {\em arXiv preprint arXiv:2202.00301}, 2022.

\bibitem[LLMS18]{LLMS18}
Mart\'{\i} Lahoz, Manfred Lehn, Emanuele Macr\`\i, and Paolo Stellari.
\newblock Generalized twisted cubics on a cubic fourfold as a moduli space of
  stable objects.
\newblock {\em J. Math. Pures Appl.}, 114(9):85--117, 2018.

\bibitem[LLSvS17]{LLSvS17}
Christian Lehn, Manfred Lehn, Christoph Sorger, and Duco van Straten.
\newblock Twisted cubics on cubic fourfolds.
\newblock {\em J. Reine Angew. Math.}, 731:87--128, 2017.

\bibitem[Log12]{Log12}
Dmitry Logachev.
\newblock Fano threefolds of genus 6.
\newblock {\em Asian J. Math.}, 16(3):515--559, 2012.

\bibitem[LPZ18]{li2018twisted}
Chunyi Li, Laura Pertusi, and Xiaolei Zhao.
\newblock {Twisted cubics on cubic fourfolds and stability conditions}.
\newblock {\em arXiv preprint arXiv:1802.01134}, 2018.

\bibitem[LPZ20]{li2020elliptic}
Chunyi Li, Laura Pertusi, and Xiaolei Zhao.
\newblock {Elliptic quintics on cubic fourfolds, O'Grady 10, and Lagrangian
  fibrations}.
\newblock {\em arXiv preprint arXiv:2007.14108}, 2020.

\bibitem[LZ21]{liu2021note}
Zhiyu Liu and Shizhuo Zhang.
\newblock {A note on Bridgeland moduli spaces and moduli spaces of sheaves on
  {$X_{14}$} and {$Y_3$}}.
\newblock {\em arXiv preprint arXiv:2106.01961}, 2021.

\bibitem[Mac07]{Mac07}
Emanuele Macr\`\i.
\newblock Stability conditions on curves.
\newblock {\em Math. Res. Lett.}, 14(4):657--672, 2007.

\bibitem[PPZ19]{perry2019stability}
Alexander Perry, Laura Pertusi, and Xiaolei Zhao.
\newblock {Stability conditions and moduli spaces for Kuznetsov components of
  Gushel--Mukai varieties}.
\newblock {\em To appear in Geom. Topol., arXiv:1912.06935}, 2019.

\bibitem[PPZ22]{ppzEnriques2021}
Alexander Perry, Laura Pertusi, and Xiaolei Zhao.
\newblock Moduli spaces of stable objects in {E}nriques categories.
\newblock {\em In preparation}, 2022.

\bibitem[PY20]{PY20}
Laura Pertusi and Song Yang.
\newblock Some remarks on {F}ano threefolds of index two and stability
  conditions.
\newblock {\em To appear in IMRN, arXiv:2004.02798}, 2020.

\bibitem[SS17]{shinder2017geometry}
Evgeny Shinder and Andrey Soldatenkov.
\newblock {On the geometry of the Lehn--Lehn--Sorger--van Straten eightfold}.
\newblock {\em Kyoto J. Math.}, 57(4):789--806, 2017.

\bibitem[Voi92]{voisin1992stabilite}
Claire Voisin.
\newblock {Sur la stabilit{\'e} des sous-vari{\'e}t{\'e}s lagrangiennes des
  vari{\'e}t{\'e}s symplectiques holomorphes}.
\newblock {\em Complex projective geometry (Trieste, 1989/Bergen, 1989)},
  179:294--303, 1992.

\bibitem[Voi16]{Voi16}
Claire Voisin.
\newblock Remarks and questions on coisotropic subvarieties and 0-cycles of
  hyper-{K}\"{a}hler varieties.
\newblock In {\em K3 surfaces and their moduli}, volume 315 of {\em Progr.
  Math.}, pages 365--399. Birkh\"{a}user/Springer, 2016.

\bibitem[Voi18]{voisin2018triangle}
Claire Voisin.
\newblock Triangle varieties and surface decomposition of hyper-{K}{\"a}hler
  manifolds.
\newblock {\em arXiv preprint arXiv:1810.11848}, 2018.

\bibitem[Voi22]{Voi21}
Claire Voisin.
\newblock On the {L}efschetz standard conjecture for {L}agrangian covered
  hyper-{K}\"{a}hler varieties.
\newblock {\em Adv. Math.}, 396:Paper No. 108108, 2022.

\bibitem[Zha20]{zhang2020bridgeland}
Shizhuo Zhang.
\newblock {Bridgeland moduli spaces and Kuznetsov's Fano threefold conjecture}.
\newblock {\em arXiv preprint arXiv:2012.12193}, 2020.

\end{thebibliography}

\bibliographystyle{alpha}

\end{document}